\documentclass[11pt]{amsart}
\usepackage{latexsym}
\usepackage{amsxtra}
\usepackage{verbatim}
\usepackage{color}
\usepackage{amsthm,amsmath,amsfonts,amssymb,mathrsfs,amscd,graphics,amscd}
\usepackage{amssymb}
\usepackage{appendix}
\usepackage{hyperref,supertabular}
\usepackage{ifpdf}
\usepackage{enumerate}
\usepackage{fancyhdr}
\usepackage{mathrsfs}
\usepackage[mathscr]{eucal}
\usepackage{upgreek}
 \usepackage{wasysym}
\usepackage{ esint }
\usepackage{pstricks}
\usepackage{color}
\usepackage{cancel}
\usepackage{verbatim}
\usepackage[all]{xy}
\usepackage{lipsum}
\usepackage{float}
\usepackage[margin=1in]{geometry}
\usepackage{longtable}
\allowdisplaybreaks

\usepackage{tikz}

\newtheorem{theorem}{Theorem}[section]
\newtheorem{lemma}[theorem]{Lemma}
\newtheorem{proposition}[theorem]{Proposition}
\newtheorem{corollary}[theorem]{Corollary}
\newtheorem{definition}[theorem]{Definition}
\newtheorem{example}[theorem]{Example}
\newtheorem{remark}[theorem]{Remark}
\newtheorem{conjecture}[theorem]{Conjecture}

\newcommand{\virt}{\mathrm{vir}}
\newcommand{\M}{\mathcal{M}}
\newcommand{\cM}{\mathcal{M}}
\newcommand{\cD}{\mathcal{D}}
\newcommand{\D}{\mathcal{D}}
\newcommand{\cL}{\mathcal{L}}
\newcommand{\cW}{\mathcal{W}}
\newcommand{\C}{\mathbb{C}}

\newcommand{\LD}{\Big\langle}
\newcommand{\RD}{\Big\rangle}
\newcommand{\one}{{\bf 1}}

\def\bmu{{\boldsymbol \mu}}

\newcommand{\CC}{\mathbb{C}}
\def\sta{^\ast}
\def\ti{\tilde}
\def\gm{\CC\sta}
\def\bd{{\mathbf d}}
\newcommand{\QQ}{\mathbb{Q}}
\newcommand{\Si}{\Sigma}
\def\sC{{\mathscr C}}
\def\sub{\subset}
\def\sN{{\mathscr N}}
\def\sL{{\mathscr L}}
\def\and{\quad{\rm and}\quad}
\def\lggd{_{g,\gamma,\bd}}
\def\cW{{\mathcal W}}
\newcommand{\pre}{ {\mathrm{pre}} }
\DeclareMathOperator{\Aut}{Aut}

\newcommand{\Obs}{ \mathrm{Obs} }
\def\cO{{\mathcal O}}
\def\Ob{\cO b}
\def\cWgg{\cW\lggd}
\newcommand{\Ga}{\Gamma}
\def\fixW{\cW^{\,T}}
\def\redd{{\mathrm{red}}}

\def\sV{{\mathscr V}}
\def\lorho{_{^{(1,\rho)}}}
\def\lophi{_{^{(1,\varphi)}}}
\newcommand{\bL}{\mathbf{L}}
 \DeclareMathOperator{\Ext}{Ext}
  
\newcommand{\Def}{ \mathrm{Def} }
\def\umv{^{\mathrm{mv}}}
\newcommand{\vir}{ {\mathrm{vir}} }
\def\Wfix{\cW_\Ga}
\newcommand{\ZZ}{\mathbb{Z}}
\newcommand{\PP}{\mathbb{P}}
\def\ft{\mathfrak{t}}
\def\cC{{\mathcal C}}
\def\cL{{\mathcal L}}
\def\cN{{\mathcal N}}
\newcommand{\EE}{\mathbb{E}}
\newcommand{\ep}{\epsilon}
\newcommand{\vGa}{\Ga} 

\newcommand{\Mbar}{\overline{\cM}}

\def\sC{{\mathscr C}}

\def\sN{{\mathscr N}}
\def\sL{{\mathscr L}}

\def\sV{\mathscr{V}}

\def\bl{\bigl(}
\def\br{\bigr)}

\begin{document}

\title
[A genus-one FJRW invariant via two methods]
{A genus-one FJRW invariant via two methods}

\author{Jun Li}
\author{Wei-Ping Li}
\author{Yefeng Shen}
\author{Jie Zhou}

\begin{abstract}
We calculate a genus-one FJRW invariant of an LG pair $(W_3=x_1^3+x_2^3+x_3^3, \mu_3)$ via two different methods.
In the first method, we apply the cosection localization technique to get a genus-one three-spin virtual class explicitly and then calculate the target FJRW invariant via self-intersections of the three-spin virtual class.
In the second method, we apply the Mixed Spin P-fields method for the pair and calculate the invariant using the localization formula.
This invariant is the building block in establishing the all-genera LG/CY correspondence and its determination 
enables one to compute the all-genera FJRW invariants for the LG pair.
\end{abstract}


\maketitle

{
\setcounter{tocdepth}{2} 
\tableofcontents
}

\section{Introduction}
Let $(W, G)$ be a {\em Landau-Ginzburg (LG) pair}, 
where $W:\C^N\to \C$ is a quasi-homogeneous polynomial with critical point only at the origin and $G$ is a nontrivial abelian group of diagonal symmetries of $W$.
In \cite{FJR07, FJR13}, Fan, Jarvis, and Ruan constructed Gromov-Witten type invariants for the pair $(W, G)$ as intersection numbers on moduli space of $W$-spin structures. 
These invariants, called {\em FJRW invariants} nowadays, 
are of the form 
$\big\langle\alpha_1\psi_1^{k_1}, \cdots, \alpha_n\psi_n^{k_n}\big\rangle_{g,n}^{(W,G)}\,.$
Here $\psi_i$'s are psi-classes pulled back from the moduli space of stable pointed-curves $\overline{\M}_{g,n}$, 
and $\alpha_i$'s are elements in the {\em FJRW vector space}  $H_{(W,G)}$, 
which is isomorphic to an {\em orbifold Jacobian algebra}
$$H_{(W,G)}\cong\bigoplus_{\gamma\in G}H_\gamma:=\bigoplus_{\gamma\in G}\left({\rm Jac}(W_\gamma)\cdot\Omega_\gamma\right)^G\,.$$
Here ${\rm Fix}(\gamma)$ is the $\gamma$-fixed locus in $\C^N$, ${\rm Jac}(W_\gamma)$ is the {\em Jacobian algebra} of $W_\gamma:=W|_{{\rm Fix}(\gamma)}$, and $\Omega_\gamma$ is the standard top form on ${\rm Fix}(\gamma)$.
If the fixed locus of $\gamma\in G$ is the origin $(0,\cdots,0)\in \C^N$, then the sector $H_\gamma$ and its elements are called {\em narrow}. 
The corresponding space is one-dimensional and we denote its standard generator by $\one_\gamma$.  Otherwise $H_\gamma$ and its elements are called {\em broad}. 

The studies of FJRW invariants in positive genus with non-semisimple Frobenius manifolds are under fast development recently. 
Various techniques have been developed for the computations, including cosection localization \cite{CLL}, 
MSP fields \cite{CLLL1, CLLL2} and NMSP fields \cite{CGLL}, 
wall-crossing formulae \cite{GR19, ClJR, Zho20}, 
log-GLSM models \cite{CJR, CJRS}, etc.
In general, it is rather difficult to compute even a single FJRW invariant in positive genus.\\

\paragraph*{\bf Main result}
In this paper, we compute a  genus-one FJRW invariant for the LG pair 
\begin{equation}
\label{cubic-pair}
(W_3=x_1^3+x_2^3+x_3^3\,, G=\langle J\rangle)\,,
\end{equation}
using the original definitions in \cite{FJR13, CLL} and the MSP method \cite{CLLL1, CLLL2}.
Here the element $J$ in the LG pair $(W_3, \langle J\rangle)$ is the automorphism of $W$ that acts on each $x_i$ by multiplication by a factor of $e^{2\pi\sqrt{-1}/3}$. Thus $\langle J\rangle\cong \mu_3.$
The corresponding FJRW space is 
$$H_{(W_3, \langle J\rangle)}\cong \C\{\one_J\,, \one_{J^2}\,, 1\cdot {\rm d}x_1\wedge {\rm d}x_2\wedge {\rm d}x_3,\, x_1x_2x_3\cdot {\rm d}x_1\wedge {\rm d}x_2\wedge {\rm d}x_3\}\,.$$
Here the elements $\one_J\in H_J$ and $\one_{J^2}\in H_{J^2}$ are narrow. 
The main result of this paper is 
\begin{theorem}
\label{main-theorem}
The genus-one FJRW invariant $\langle\one_{J^2}, \one_{J^2}, \one_{J^2}\rangle_{1,3}^{(W_3,\langle J\rangle)}$ 
is given by
\begin{equation}
\label{main-inv}
\langle \one_{J^2}, \one_{J^2}, \one_{J^2}\rangle_{1,3}^{(W_3,\langle J\rangle)}={1\over 108}
\,.
\end{equation}
\end{theorem}
The definitions of this invariant \cite{FJR13, CLL} will be reviewed in the body of the paper.
We shall provide two different proofs of Theorem \ref{main-theorem}. 
One relies on the original definitions, the other is based on the Mixed Spin P-field theory developed in \cite{CLLL1, CLLL2}. 
Note that this genus-one invariant can also be computed by the wall-crossing formulae, see \cite{GR19} for the case of Fermat quintic polynomial.

Theorem \ref{main-theorem} is crucial to the work \cite{LSZ20}. 
Based on Theorem \ref{main-theorem},
there the authors found that the sequence of invariants
$\{\langle\one_{J^2}, \one_{J^2}, \cdots ,\one_{J^2}\rangle_{1,\ell}^{(W_3,\langle J\rangle)}\}_{\ell\geq 0}$
matches the Taylor coefficients of the Eisenstein series $E_2(\tau)$
expanded at a particular interior point on the upper-half plane.
Then an all-genera LG/CY correspondence between the FJRW theory of the pair $(W_3, \langle J\rangle)$ and the Gromov-Witten theory of the elliptic curve given as the hypersurface $(W_3=0)\subseteq\mathbb{P}^2$ is established. 
This provides an approach to compute the higher genus FJRW invariants of the LG pair from the higher genus Gromov-Witten invariants of the cubic elliptic curve.\\

\paragraph*{\bf Plan of the paper}
The paper is organized as follows. 
In Section \ref{sec:threespin}
we define and compute the invariant in Theorem \ref{main-theorem}
by the cosection localization technique applied to the three-spin moduli.
In Section \ref{sec:reviewMSP} we review the set-up of MSP field theory for the cubic curve case.
In Section \ref{sec:MSPlocalization} we present the second method to calculate the invariant using the localization formula in MSP theory.
Some length computations are relegated to Appendix \ref{sec:appendices}.\\

\paragraph*{\bf Acknowledgment} 
Y.~Shen thank Felix Janda, Drew Johnson, and Aaron Pixton for helpful discussion on intersection theory.
We thank Huai-Liang Chang and Yang Zhou for useful discussions.

J. ~Li was partially supported by NSF grant DMS-1564500 and DMS-1601211. 
W.-P. ~Li is partially supported by Hong Kong GRF16303518 and GRF16304119.
Y. ~Shen is partially supported by Simons Collaboration Grant 587119. 
J. ~Zhou is supported by a start-up grant at Tsinghua University.


\section{A genus-one FJRW invariant via three-spin curves}
\label{sec:threespin}

We work with the field of complex numbers $\C$ and take $\mathbb{G}_m=\mathrm{GL}_{1}(\C)$.
The FJRW theory is originally constructed for any LG pair $(W, G)$ in \cite{FJR07, FJR13}, where $W: \C^n\to \C$ is a quasi-homogeneous polynomial with isolated critical point only at the origin, and $W$ contains no $x_ix_j$ terms for $i\neq j$.
The variables $x_1, \dots, x_n$ have rational weights $q_1, \cdots, q_n\in (0, {1\over 2}]\cap\mathbb{Q}$.
The group $G$ in the pair is a subgroup of the {\em group of diagonal symmetries of $W$}
$$G_{\rm max}=\left\{(\lambda_1,\cdots,\lambda_n)\in \mathbb{G}_m^n \mid W(\lambda_1\cdot x_1, \cdots, \lambda_n\cdot x_n)=W(x_1, \cdots, x_n)\right\},$$
containing the {\em exponential grading element} defined in \eqref{exponential-element} below.
The original set-up  in \cite{FJR07} uses analytic tools. The theory was later partially reformulated using the cosection localization method in \cite{CLL} for the subspace of narrow elements. 
In \cite{CLL}, the construction depends on the group $G<\mathbb{G}_m^n$ only. 
The essential ingredient is to construct the virtual classes for the moduli of the so-called $G$-spin curves.

In this paper, we follow the set-up in \cite{CLL}.
The goal of this section is to define and compute the FJRW invariant $\langle\one_{J^2}, \one_{J^2}, \one_{J^2}\rangle_{1,3}^{(W_3, \langle J\rangle)}$ for the LG pair $(W_3=x_1^3+x_2^3+x_3^3, G=\langle J\rangle)$.

\subsection{Witten's top Chern class via cosection localization}
\subsubsection{Witten's top Chern class via cosection localization}
We now briefly review  the construction of Witten's top Chern class via cosection localization for the LG space $([\C^n/ G], W)$ in \cite{CLL}.

Let $d\in \mathbb{Z}_+$ and $\delta=(\delta_1, \cdots, \delta_n)\in \mathbb{Z}_+^n$ be a primitive $n$-tuple. 
Let $\zeta_d:=\exp\left({2\pi\sqrt{-1}/d}\right)$ be the  primitive  $d$-th root of unity.
A finite subgroup $G< \mathbb{G}_m^n$ is called a $(d, \delta)$-group if it contains
\begin{equation}
\label{exponential-element}
J=(\zeta_d^{\delta_1}, \cdots, \zeta_d^{\delta_n})\in (\mu_d)^n< \mathbb{G}_m^n.
\end{equation}
For a $(d, \delta)$-group $G<\mathbb{G}_m^n$, the group 
$$\Lambda_{G}=\{{\bf m}=(m_1, \cdots, m_n)\mid {\bf x}^{\bf m}=\prod_ix_i^{m_i}~ \text{is a}~G\text{-invariant Laurent monomial in } (x_1, \cdots, x_n)\}$$
is a free Abelian group of rank $n$. 
Fixing a set of generators $\{{\bf m}_1, \cdots, {\bf m}_n\}$, we obtain a Laurent polynomial $W=\sum\limits_{i=1}^{n} {\bf x}^{{\bf m}_i}$ in $n$ variables. 

Fix integers $g$, $\ell$, and a collection of elements $\gamma=(\gamma_1, \cdots, \gamma_\ell)\in G^{\ell}$
where each $\gamma_i$ is narrow (that is, the fixed locus of each $\gamma_i$ is the origin ${\bf 0}\in \C^n$).
We consider the moduli stack of {\em $G$-spin $\ell$-pointed genus-$g$ twisted nodal curves banded by $\gamma$}, denoted by
$$\Mbar_{g, \gamma}(G)=\left\{(\mathcal{C}, \cL_1, \cdots, \cL_n) \mid {\bf m}_k(\cL_1, \cdots, \cL_n)\xrightarrow{\cong}
\left(\omega_{\mathcal{C}}^{\rm log}\right)^{w({\bf m}_k)}, \quad k=1, \cdots, n\right\}\,.$$
Here $\mathcal{C}$ is a stable $\ell$-pointed genus-$g$ twisted nodal curve; $\cL_j$'s are invertible sheaves on $\mathcal{C}$ such that the representations of $\cL_j$ restricted to the marked points of $\mathcal{C}$ are given by the collection $\gamma\in G^{\ell}$;
$\omega_{\mathcal{C}}^{\rm log}$ is the log-dualizing sheaf of $\mathcal{C}$;
the multiplication in the Laurent polynomial ${\bf m}_k(\cL_1, \cdots, \cL_n)$ is the tensor product of invertible sheaves;
and
$$w({\bf m})=d^{-1}\cdot \deg\left( {\bf m} (t^{\delta_1}, \cdots, t^{\delta_n})\right)\in \mathbb{Z}\,.$$
The moduli stack is independent of the choice of the set of generators for $\Lambda_{G}$.

In \cite{CLL}, the authors introduced the moduli space with $p$-fields
$$\Mbar_{g,\gamma}(G)^p=\left\{[\cC,\cL_j, \rho_j]_{i=1}^{n} \mid [\cC, \cL_1, \cdots, \cL_n] \in \Mbar_{g, \gamma}(G)\,, \rho_j\in \Gamma(\cL_j)\right\}.$$
It is  a quasi-projective Deligne-Mumford stack. 
Since $\gamma$ is a collection of narrow elements,
the $G$-invariant polynomial $W$ induces a cosection from the obstruction sheaf $\mathcal{O}b_{\Mbar_{g,\gamma}(G)^p}$ to the structure sheaf $\mathcal{O}_{\Mbar_{g,\gamma}(G)^p}$.
Applying the cosection localized virtual class of Kiem-Li \cite{KL}, a virtual class  is constructed for the moduli stack $\Mbar_{g, \gamma}(G)$ in \cite{CLL}. 
This virtual class is called the {\em Witten's top Chern class} for the LG space $([\C^n/ G], W)$. 
It generalizes the cases for $r$-spin curves where the LG space is $([\C/ \mu_r], x^r)$, with $\mu_r< \mathbb{G}_m$ the multiplicative group of $r$-th roots of the unity.
The cosection localized virtual class of the moduli space $\Mbar_{g,\gamma}(G)$ is also independent of the choice the generators of $\Lambda_{G}$ that give rise to
 the $G$-invariant polynomial $W$. We denote this virtual class by 
$$[\Mbar_{g,\gamma}(G)^p]^{\virt}\in A_*(\Mbar_{g,\gamma}(G))\,.$$

\subsubsection{Fermat cubic polynomial in three variables}
We will work with $G=\langle J\rangle< \mathbb{G}_m^3$, which is a $(d,\delta)=(3, (1,1,1))$-group. 
We have 
\begin{equation}
G=\langle J\rangle =\langle (\zeta_3, \zeta_3, \zeta_3)\rangle\cong \mu_3< (\mu_3)^3 <  \mathbb{G}_m^3\,.
\end{equation}
To indicate the role of the polynomial $W$ in the construction of the cosection localized virtual class, we always include the special polynomial $W_3=x_1^3+x_2^3+x_3^3$ in the notation. 
Then we have 
\begin{equation}
\label{fermat-cubic-virt}
[\Mbar_{g,\gamma}(W_3, \mu_3)^p]^{\virt}\in A_*(\Mbar_{g,\gamma}(W_3, \mu_3))\,.
\end{equation}

A narrow element is of the form $\gamma_i=J^{k_i} \in G$, $k_i=1$ or $2$. For simplicity, we denote
$$\gamma=(\zeta_3^{k_1}, \cdots, \zeta_3^{k_\ell}):=(k_1, \cdots, k_\ell)\,.$$
The moduli of three-spin curves is of crucial importance to us, where the LG pair is $(x^3, \mu_3)$ with $\mu_3< \mathbb{G}_m$.
Let $\pi_1: G=\langle J\rangle\cong\mu_3\to\mu_3$ be the projection to the first component.
It is clearly an isomorphism. Therefore for the moduli of three-spin curves, we can also use $\gamma$ to denote $\pi_1(\gamma)$ without ambiguity.
For simplicity, we denote the three-spin moduli by 
$$
\Mbar_{g, \gamma}^{1/3}:=\Mbar_{g,\gamma}(x^3, \mu_3)\,.
$$
Its cosection localized virtual class is denoted by
\begin{equation}
\label{cubic-spin-virt}
[\Mbar_{g, \gamma}^{1/3, p}]^\virt:=[\Mbar_{g,\gamma}(x^3, \mu_3)^p]^{\virt}\in A_*(\Mbar_{g,\gamma}(x^3, \mu_3))=A_*(\Mbar_{g, \gamma}^{1/3})\,.
\end{equation}

Let us compare these moduli spaces $\Mbar_{g,\gamma}(W_3, \mu_3)$, $\Mbar_{g, \gamma}^{1/3}$,  and the virtual classes \eqref{fermat-cubic-virt}, \eqref{cubic-spin-virt}.
For the case $\mu_3\cong G<\mathbb{G}_m^3$, the free group $\Lambda_{G}$ is generated by $x_1^3, x_1x_2^{-1}$ and $x_1x_3^{-1}$. Thus
\begin{equation}
\label{cubic-moduli}
\Mbar_{g, \gamma}(W_3, \mu_3)=\{[\cC, \cL_1, \cL_2, \cL_3]\mid \cL_1^{\otimes 3}\cong \omega_{\cC}^{\rm log}\,,
\quad \cL_1\otimes\cL_2^{-1}\cong\cL_1\otimes \cL_3^{-1}\cong\mathcal{O}_{\cC}\}\,.
\end{equation}
Here $\cO_{\cC}$ is the structure sheaf of $\cC$.
For the three-spin case $\mu_3<\mathbb{G}_m$, we have
\begin{equation}
\label{three-spin-moduli}
\Mbar_{g, \gamma}^{1/3}
=\{[\cC, \cL]\mid  \cL^{\otimes 3}\cong \omega_{\cC}^{\rm log}\}\,.
\end{equation}
Since $\gamma$ is narrow, the two moduli spaces are isomorphic via the morphism
$$F:\, \Mbar_{g, \gamma}(W_3, \mu_3)\longrightarrow\Mbar_{g, \gamma}^{1/3}\,,
$$
which maps $[\cC, \cL_1, \cL_2, \cL_3]$ to $[\cC, \cL_1]$.


\subsubsection{From three-spin to Fermat cubic}
In our case, we take 
\begin{equation}
\label{basic-setting}
g=1\,,
\quad
\gamma=(2,2,2):=(2^3)\,.
\end{equation}
The virtual classes $\left[\Mbar_{1, (2^3)}(W_3, \mu_3)^p\right]^\virt\in A_0(\Mbar_{1, (2^3)}(W_3, \mu_3))$ and 
$
\left[\Mbar_{1, (2^3)}^{1/3, p}\right]^\virt
\in A_2(\Mbar_{1, (2^3)}^{1/3})$ are different but closely related. The following result is obtained in \cite[Lemma 2]{LSZ20}.
\begin{proposition}[\cite{LSZ20}]
The virtual class of the moduli stack $\overline{\M}_{1,(2^3)}(W_3, \mu_3)$ is 
the triple self-intersection of that of $\overline{\cM}^{1/ 3}_{1, (2^3)}$.
That is,
\begin{equation}
\label{3-spin-cubic}
\left[\overline{\M}_{1,(2^3)}(W_3, \mu_3)^p\right]^{\rm vir}=
\left(\left[
\overline{\mathcal{M}}_{1, (2^3)}^{1/3,p}\right]^{\rm vir}
\right)^3
\in A_{0}
\left(\Mbar_{1, (2^3)}^{1/3}\right)\,.
\end{equation}
\end{proposition}
Our target FJRW invariant in Theorem \ref{main-theorem} is defined to be the degree of the virtual class
\begin{equation}\label{eqnalgebraicdfn}
\langle\one_{J^2}, \one_{J^2}, \one_{J^2}\rangle_{1,3}^{(W_3,\langle J\rangle)}
:= \deg\left(\left[\overline{\M}_{1,(2^3)}(W_3, \mu_3)^p\right]^{\rm vir }\right)\,.
\end{equation}
which we shall also denote by $\Theta_{1,3}$ for convenience of notations.
This invariant belongs to the class of so-called {\em primitive FJRW invariants} \cite{CLLL2}. More generally, we have a sequence of genus-one primitive FJRW invariants
$$\Theta_{1,\ell}:=\langle\underbrace{\one_{J^2}, \cdots, \one_{J^2}}_{\ell\text{-tuple}}\rangle_{1,\ell}^{(W_3,\mu_3)}
:=\deg\left(\left[\overline{\M}_{1,(2^\ell)}(W_3,  \mu_3)^p\right]^{\rm vir }\right), \quad \ell\geq1\,.$$
If $3\nmid\ell$, the invariant $\Theta_{1,\ell}$ vanishes as the moduli space $\overline{\M}_{1,(2^\ell)}(W_3, \mu_3)$ is empty.

\subsubsection{Comparison between FJRW invariants and primitive FJRW invariants}
Now we compare the primitive FJRW invariants here with the invariants originally constructed in \cite{FJR07, FJR13}.
The original method uses both algebraic tools and analytic tools. For any $\alpha_i\in H_{\gamma_i}$, with $\gamma_i\in G$ not necessarily narrow,
the primary FJRW invariant for the LG pair $(W, G)$ in \cite[Definition 4.2.6]{FJR13}, which we denote by $\langle \alpha_1, \cdots, \alpha_\ell\rangle_{g,\ell}^{\rm FJRW}$, is defined
 to be an integral
(\cite[Definition 4.2.1]{FJR13}) 
$$
\int_{\Mbar_{g,\ell}}
\Lambda_{g,\ell}^{\rm FJRW}(\alpha_1, \cdots, \alpha_\ell)=
\int_{\Mbar_{g,\ell}}
{|G|^{g}\over \deg{\rm st}} \, {\rm PD}\, {\rm st}_*\left(\left[\overline{\cW}_{g, \ell}(W, \gamma)\right]^{\rm vir}\cap\prod_{i=1}^{\ell}\alpha_i\right). 
$$
Here ${\rm st}$ is the forgetful morphism from the moduli of $W$-spin curves $\overline{\cW}_{g, \ell}(W, \gamma)$ to $\Mbar_{g,\ell}$. 
When $\gamma$ is a collection of narrow elements, $\overline{\cW}_{g, \ell}(W, \gamma)$ is the same as the moduli of $G$-spin curves $\Mbar_{g,\gamma}(G)$. 
Let $\alpha_i=\one_{\gamma_i}$, then the FJRW virtual class $\left[\overline{\cW}_{g, \ell}(W, \gamma)\right]^{\rm vir}\cap\prod_{i=1}^{\ell}\alpha_i$ differs from the cosection localized virtual class $[\Mbar_{g,\gamma}(G)^p]^{\virt}$ by a sign $(-1)^{\varepsilon(\gamma)}$, with $\varepsilon(\gamma)=\sum_i{\rm rank}(R^\bullet\pi_*\cL_i)$. 
More precisely, according to \cite[Theorem 5.6]{CLL}, we have
$$
[\Mbar_{g,\gamma}(G)^p]^{\virt}=(-1)^{\varepsilon(\gamma)}\cdot 
\left[\overline{\cW}_{g, \ell}(W, \gamma)\right]^{\rm vir}\cap\prod_{i=1}^{\ell}\alpha_i\in H_*(\Mbar_{g,\gamma}(G))\,.
$$
In the current setting \eqref{basic-setting}, 
$|G|=3, g=1, \deg{\rm st}=|G|^{2g-1}=3,$ and $\varepsilon(\gamma)=1.$
We have
\begin{equation}
\label{degree-equality}
\deg\left(\left[\overline{\M}_{1,(2^3)}(W_3, \mu_3)^p\right]^{\rm vir }\right)
=\deg\left({\rm st}_*\left[\overline{\M}_{1,(2^3)}(W_3, \mu_3)^p\right]^{\rm vir }\right),
\end{equation}
and thus
$$\langle\one_{J^2}, \one_{J^2}, \one_{J^2}\rangle_{1,3}^{(W_3,\langle J\rangle)}=-\langle\one_{J^2}, \one_{J^2}, \one_{J^2}\rangle_{1,3}^{\rm FJRW}\,.$$

\subsection{Intersection theory on $\overline{\cM}_{1,3}$}
Before we compute the FJRW invariant $\Theta_{1,3}$, we briefly review the intersection theory on $\overline{\cM}_{1,3}$. 
Following \cite{AC98}, we label the boundary classes in $A^1(\overline{\cM}_{1,3})$ by $\delta_{0,3}, \delta_{0,2}$, and $\delta_{\rm irr}$, 
where
\begin{itemize}
\item $\delta_{0,k},k=2,3$ consists of stable genus-one curves
whose partial normalizations at a node are the unions of two connected components 
of genus $1$ with $(3-k)$ markings and of genus $0$ with $k$ markings, respectively.
\item $\delta_{\rm{irr}}$ consists of stable genus-one curves
whose partial normalizations at a node are connected.
\end{itemize}
The dual graphs of these classes are illustrated schematically in Figure \ref{figureboundarydivisors} below.
The genus here represents the geometric genus, that is, the genus of the normalization.

\begin{figure}[h]
  \renewcommand{\arraystretch}{1} 
\begin{displaymath}
\includegraphics[scale=0.5]{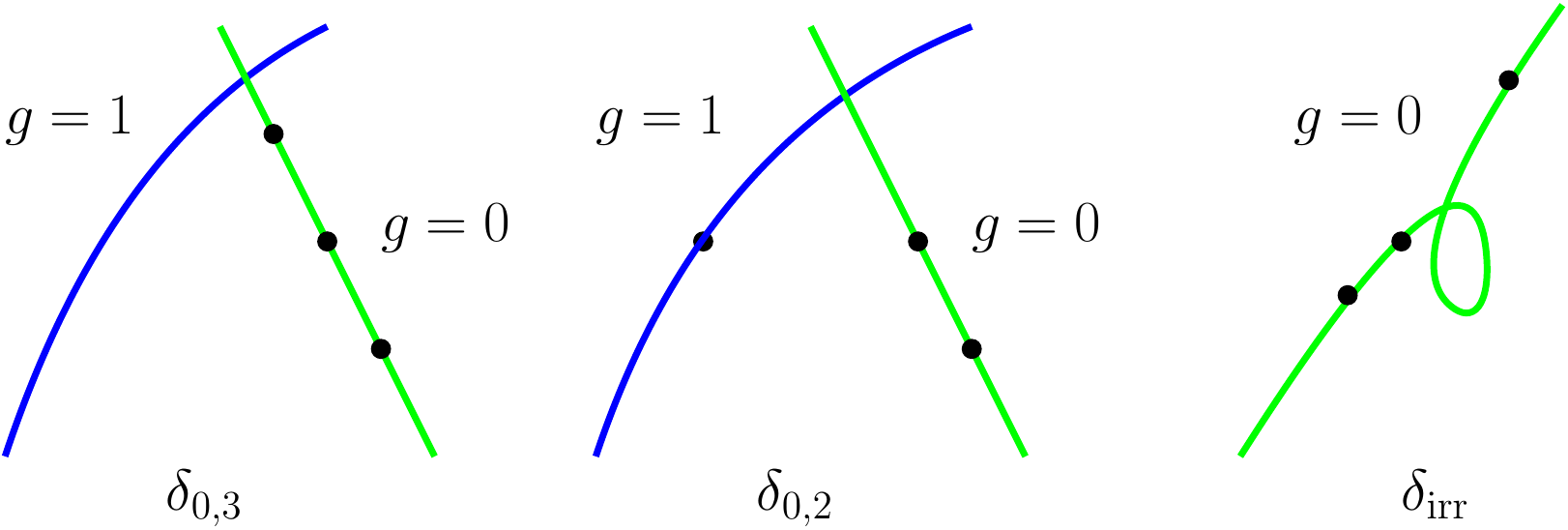}
\end{displaymath}
  \caption[boundarydivisors]{Configurations in codimension-one boundary divisors on $\overline{\mathcal{M}}_{1,3}$.}
  \label{figureboundarydivisors}
\end{figure}

Let $S_n$ be the symmetric group of $n$ letters. We are interested in the $S_3$-equivariant part of $H^*(\overline{\mathcal{M}}_{1,3}, \mathbb{Q})$.
According to \cite{G98}, the $S_3$-equivariant part of both $H^2(\Mbar_{1,3}, \mathbb{Q})$ and $H^4(\Mbar_{1,3}, \mathbb{Q})$ are of three dimensional. 
We summarize the results here:
\begin{itemize}
\item The $S_3$-equivariant part of $H^2(\Mbar_{1,3}, \mathbb{Q})$ is spanned by $\delta_{\rm irr}$, $\delta_{0,2},$ and $\delta_{0,3}$. 
\item The $S_3$-equivariant part of $H^4(\Mbar_{1,3}, \mathbb{Q})$ is spanned by $\delta_{0,2}\delta_{0,3}$, $\delta_{\rm irr}\delta_{0,3}$, and $\delta_{\rm irr}\delta_{0,2}$. 
\end{itemize}

Some intersection numbers on $\overline{\cM}_{1,3}$ are listed in Table  \ref{table-cubic} below. 
The results can be obtained from direct calculation 
or Sage programs, for instance, Drew Johnson's Sage program {\em strataalgebra}, see \href{https://pypi.org/project/mgn/}{https://pypi.org/project/mgn/}. 

\begin{table}[h] 
\caption{Intersection numbers on $\overline{\mathcal{M}}_{1,3}$\,.}
  \label{table-cubic}
    \renewcommand{\arraystretch}{1.5} 
    \centering
 \begin{tabular}{|c|ccc|cc|}
 \hline
 &$\delta_{0,2}\delta_{0,3}$&$\delta_{\rm irr}\delta_{0,3}$&$\delta_{\rm irr}\delta_{0,2}$&$\delta_{0,2}^2$&$\delta_{0,3}^2$
 \\
 \hline
$\delta_{\rm irr}$&${3\over2}$&0&0&$-{3\over 2}$&$-{1\over 2}$
\\
$\delta_{0,2}$&0&${3\over2}$&-${3\over2}$&${1\over 8}$&$-{1\over 8}$
\\
$\delta_{0,3}$&-${1\over8}$&-${1\over2}$&${3\over2}$&$0$&${1\over 12}$
\\
\hline
\end{tabular}
\end{table} 

We illustrate the table by considering the triple intersection number $\delta_{\rm irr}\cdot\delta_{0,2}\cdot\delta_{0,3}$.
The dual graph is depicted  in
Figure \ref{figuretripleintersections} below. 
\begin{figure}[H]
  \renewcommand{\arraystretch}{1} 
\begin{displaymath}
\includegraphics[scale=0.5]{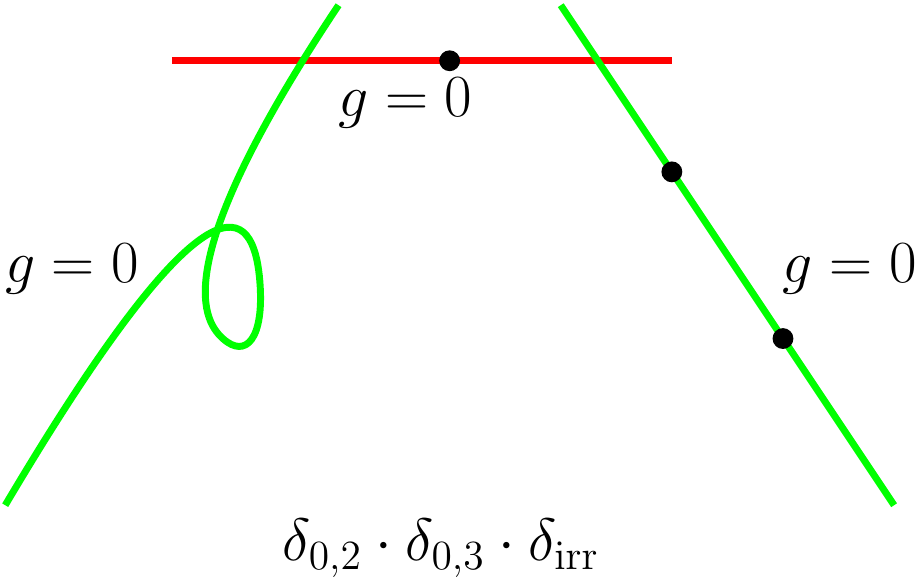}
\end{displaymath}
  \caption[tripleintersections]{Configuration in triple intersection.}
  \label{figuretripleintersections}
\end{figure}
We have 
$$\delta_{\rm irr}\cdot\delta_{0,2}\cdot\delta_{0,3}=3\cdot{1\over 2}\cdot 1\cdot1={3\over 2}\,.$$
Here $3={3\choose 1}$ is the number of choices of the point in the middle component. The factor $2$ is the order of the automorphism of the component on the leftmost component $\delta_{\rm irr}$. 
 All the $1$'s thoughout are given by the integral
$$\int_{\Mbar_{0,3}}1=1\,.$$

Some useful cohomological relations will be used later.
For example, we have (cf. \cite{G97})
\begin{equation}
\label{vanishing-irr}
\delta_{\rm irr}^2=0\,.
\end{equation}
We also have a genus-one relation in \cite[Theorem 2.1]{AC98}  among the boundary classes, the psi-classes $\psi_i$, and  the kappa-class $\kappa_{1}$:
\begin{equation}\label{g=1-kappa}
\kappa_1-\sum_{i=1}^3\psi_i=-\delta_{0,2}-\delta_{0,3}\in H^*(\overline{\M}_{1,3})\,.
\end{equation}

\subsection{A three-spin virtual class}

In this part, we study the Witten's top Chern class $\left[\overline{\mathcal{M}}_{1, (2^3)}^{1/3,p}\right]^{\rm vir}$ in
\eqref{3-spin-cubic}
for the moduli stack of three-spin curves. 
For simplicity, we denote the moduli stack by
$$\cW:=\overline{\mathcal{M}}_{1, (2^3)}^{1/3}= \Mbar_{1,(2^3)}(x^3,\mu_3)\,.$$ 
The formula below is obtained in  \cite{LSZ20} by analyzing the cosection localized virtual class explicitly.
\begin{proposition}[\cite{LSZ20}]
\label{propvirtualW}
Let
$\pi: \cC\rightarrow \cW$ be the universal family.
Then Witten's top Chern class of the three-spin moduli is
\begin{equation}
\label{cosection-calculation}
\left[\overline{\mathcal{M}}_{1, (2^3)}^{1/3,p}\right]^{\rm vir}=-c_1(R^{\bullet}\pi_*\cL)-3[\cD]\in A^1 \cW\,.
\end{equation}
Here  $c_1(R^{\bullet}\pi_*\cL)$ is the first Chern class of the complex $R^{\bullet}\pi_*\cL$ and $\cD$
is the jumping locus of the section $\rho\in\Gamma(\cL)$.
\end{proposition}

The formula in \eqref{cosection-calculation} is a special case of a more general formulae conjectured by Janda \cite{Jan17} for all $r$-spin curve
moduli.
In general, such a virtual class may contain non-tautological classes such as $[\cD]$.\\

Let $\rm{st}:\mathcal{W}\rightarrow \overline{\mathcal{M}}_{1,3}$ be the forgetful morphism. 
The term $c_1(R^{\bullet}\pi_*\cL)$ can be calculated explicitly by Chiodo's formula \cite[Corollary 3.1.8]{Chi08}.
\begin{corollary}
The first Chern class of the complex $R^{\bullet}\pi_*\cL$ is given by
\begin{equation}
\label{first-chern}
c_1(R^{\bullet}\pi_*\cL)
=
{B_2({1\over 3})\over 2!}{\rm st}^*\kappa_1-{B_2({2\over 3})\over 2!}\sum_{i=1}^3{\rm st}^{*}\psi_i+{B_2(0)\over 2!}\D_{0,2}+{3\,B_2({1\over 3})\over 2!}
\D_{0,3}+\sum_{q=0}^{2}{B_2({q_{-}\over 3})\over 2!}
m_q\D_{{\rm irr}, q}\,.
\end{equation}
\end{corollary}
Hereafter by abuse of notation, we use the same notations as the stacks to denote the corresponding classes in the Chow ring.
We now explain the notations in \eqref{cosection-calculation} and \eqref{first-chern}:
\begin{itemize}
\item
Define $B_{2}(x)=\sum\limits_{k=0}^{2} {2\choose k}B_{k} x^{n-k}$, where $B_{k}$ is the $k$-th Bernoulli numbers. 
\item
Here
$\D_{0,k}$ is the substack in $\cW$, such that the pointed curve $\mathcal{C}$ at a generic point in the substack is a genus-one nodal curve with $3$ markings, the node is a 
separating node, and there are $k$ markings on the genus-zero component. 
The node are decorated by the local monodromies $q:=q_{+}$ on the genus-zero component and $q_{-}$ on the genus-one component. Both $q_+, q_-\in \{0, 1, 2\}$, with a balance condition $$q_{+}+q_{-}\equiv 0 \mod 3\,.$$
Similar to  $\delta_{\rm{irr}}$, $\D_{\mathrm{irr},q}$ consists of those whose curve parts are stable genus-one curves
whose partial normalizations at a node are connected, and $(q:=q_+, q_-)$ is the pair of monodromies on the pair of the branches.
In this case, $\D_{\mathrm{irr},1}=\D_{\mathrm{irr},2}.$
See Figure \ref{figuredivisors} below for an illustration.
\begin{figure}[h]
  \renewcommand{\arraystretch}{1} 
\begin{displaymath}
\includegraphics[scale=0.5]{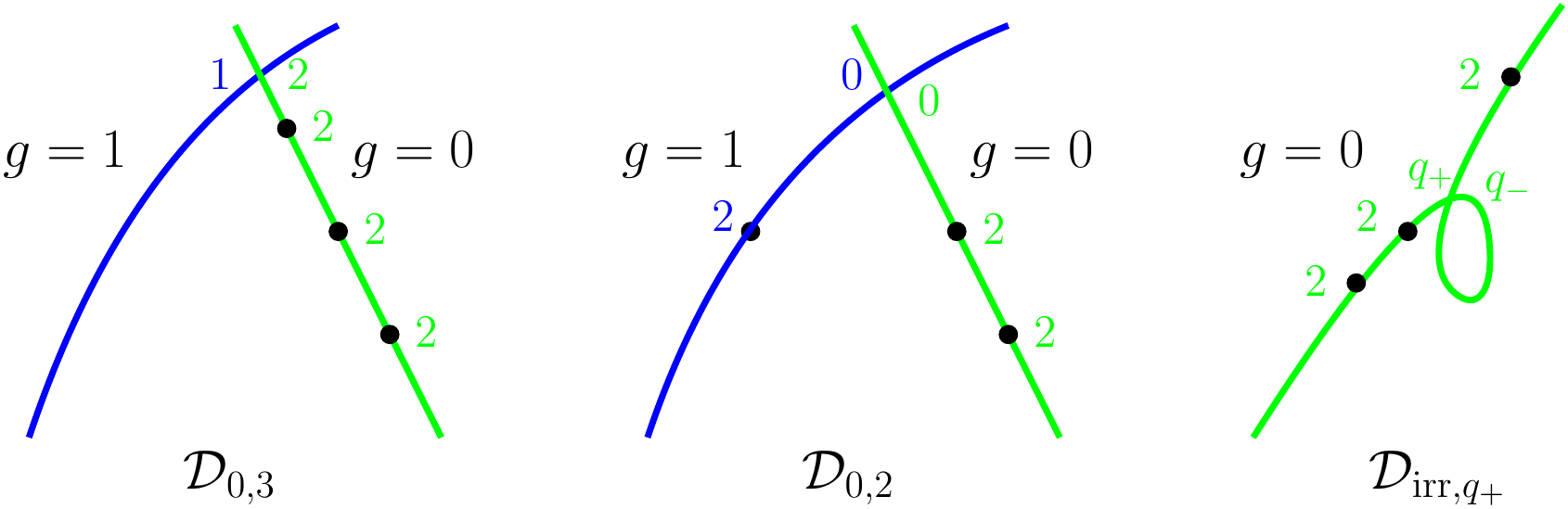}
\end{displaymath}
  \caption[divisors]{Configurations in various divisors on $\cW=\overline{\mathcal{M}}_{1, (2^3)}^{1/3}$.}
  \label{figuredivisors}
\end{figure}

\item 
The number $m_q$ is the order of the additive group $\langle q\rangle$, which is a subgroup of $\mathbb{Z}/r\mathbb{Z}$ generated by the class $[q]\in \mathbb{Z}/r\mathbb{Z}$ for $q\in \mathbb{Z}$. When $r=3$, we have $m_0=1$ and $m_1=m_2=3$.
 \end{itemize}

Due to the presence of $H^1(\cC, G)$ \cite{FJR13}, each of $\D_{0,3}$ and $\D_{0,2}$ has $3^2=9$ components. 
Let $\D_{0,3}^{(0)}$ be the component of $\cD_{0,3}$ such that $h^0(\cC,\mathcal{L})=1$ on the genus-one component of the curve. 
We label the rest eight components by $\D_{0,3}^{(i)}$, $i=1, \cdots, 8$,
respectively.
See Figure \ref{figuredivisorsD03} for an illustration.
\begin{figure}[h]
  \renewcommand{\arraystretch}{1} 
\begin{displaymath}
\includegraphics[scale=0.5]{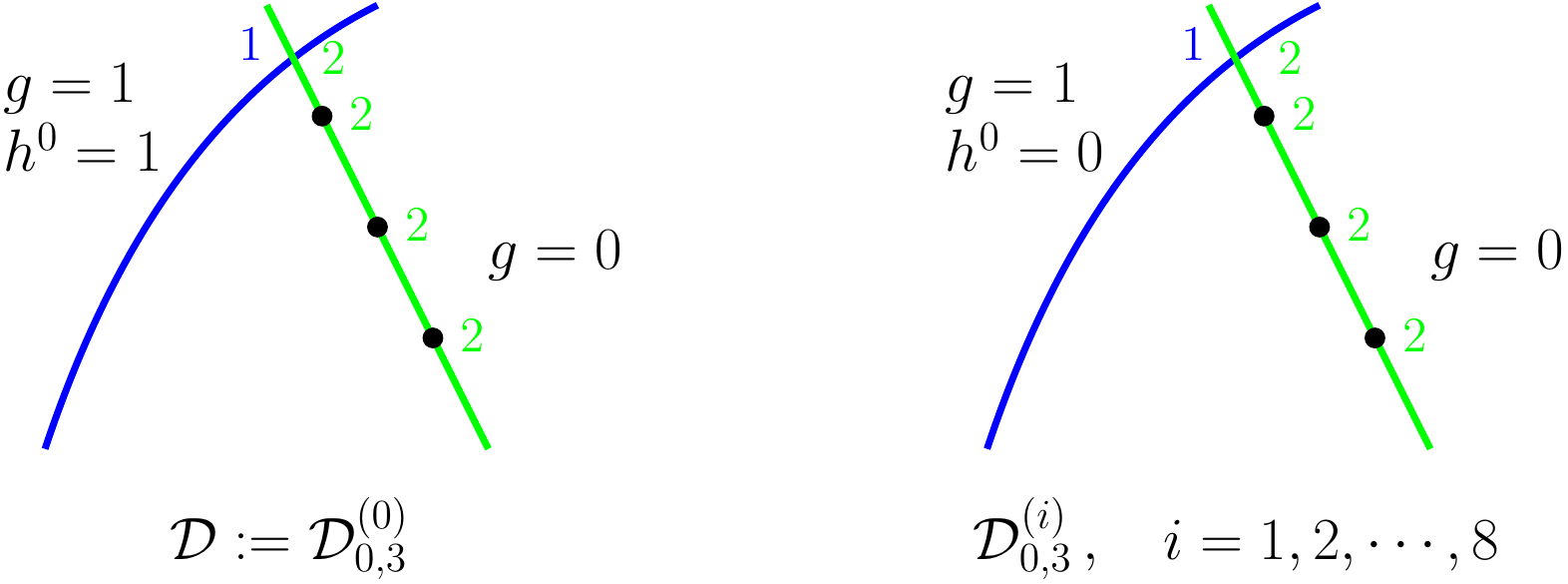}
\end{displaymath}
  \caption[divisors]{Configurations in various divisors $\cD^{(i)}_{0,3},i=0,1,\cdots, 8$ on $\cW=\overline{\mathcal{M}}_{1, (2^3)}^{1/3}$.}
  \label{figuredivisorsD03}
\end{figure}

Note that $\cD_{0,3}^{(0)}$ is exactly the closed locus on $\cW $ where $R^{0}\pi_{*}\cL\neq 0$. Thus $\cD=\cD_{0,3}^{(0)}$ in 
\eqref{cosection-calculation}.
Using \eqref{first-chern}, we see that the three-spin virtual cycle is
\begin{equation}
\label{mixed-terms}
\left[\overline{\mathcal{M}}_{1, (2^3)}^{1/3,p}\right]^{\rm vir}
={{\rm st}^*{\kappa_1}\over 36}-\sum\limits_{i=1}^{3}{{\rm st}^*{\psi_i}\over 36}-{\D_{\rm irr, 0}-\D_{\rm irr, 1}-\D_{\rm irr, 2}\over 12}-{\D_{0,2}\over 12}+{\D_{0,3}\over 12}-3 \D_{0,3}^{(0)}\,.
\end{equation}

\subsection{Push-foward and flat pull-back}
We need some preparations before carrying out the explicit calculation.
Consider the forgetful morphism  $\mathrm{st}:\Mbar_{g,\gamma}(G)\rightarrow \overline{\mathcal{M}}_{g,n}$. The number of fibers 
over a generic point in $\overline{\mathcal{M}}_{g,n}$ equals to the order of the  group $H^1(\cC, G)$. Also each fiber $(\cC, \cL)$ over a generic point has an automorphism ${\rm Aut}_{\cC}(\cL)\cong G$.
Let $\cW_{\Gamma}$ be the substack of $\cW$ associated to a dual graph $\Gamma$.
Let 
$$\deg \cW_{\Gamma}:=\deg{\rm st}_{\Gamma}:=\deg\, \mathrm{st}|_{\cW_{\Gamma}}.$$
The following result is obtained in \cite[Proposition 2.2.17, Proposition 2.2.18]{FJR13}.
\begin{lemma}
\label{lemma-degree}
If $\Gamma$ is a tree with two vertices and one edge, with each tail decorated with  the local monodromy valued in $G$ and the edge decorated with $(\gamma_+, \gamma_-)$.
Let $m_{\pm}$ be the order of $\langle{\gamma_\pm}\rangle$. Then
\begin{equation}
\label{tree-degree}
\deg{\rm st}_{\Gamma}={|G|^{2g-1}\over m_+}\,.
\end{equation}
If $\Gamma$ is a loop with a single vertex and a single edge decorated with $(\gamma_+, \gamma_-)$, then
$$\deg{\rm st}_{\Gamma}={|G|^{2g-2}\over m_+}\,.$$
\end{lemma}
 In our case, $g=1$, $|G|=3$, thus 
\begin{equation}
\label{eqnstdegree}
\deg \cW:=\deg (\mathrm{st})={1\over 3}\cdot 3^2=3\,.
\end{equation}
\begin{corollary}
\label{degree-codimension-one}
The degrees of ${\rm st}_{\Gamma}$
on the codimension-one strata are given by
$$\deg \D_{\rm irr, 0}=1, \quad \deg \D_{\rm irr, 1}=\deg \D_{\rm irr, 2}={1\over 3}, 
\quad \deg\cD_{0,3}=1, \quad \deg\cD_{0,2}=3\,,\quad
\deg\cD_{0,3}^{(0)}={1\over 9}\,.$$
\end{corollary}

Next we consider two types of codimension-two strata shown in Figure  \ref{figuredoubleintersections} below.
The degrees for the intersections in  type I can be calculated using Lemma \ref{lemma-degree} repeatedly.
\begin{corollary}
\label{degree-codimension-two}
We have 
$$\deg\cD_{0,3}\cdot\cD_{0,2}=1\,, \quad \deg\cD_{{\rm irr}, q}\cdot \cD_{0,2}={1\over m_q}\,, \quad \deg\cD_{{\rm irr}, q}\cdot\cD_{0,3}={1\over 3\cdot m_q}\,.$$
\end{corollary}


\begin{figure}[H]
  \renewcommand{\arraystretch}{1} 
\begin{displaymath}
\includegraphics[scale=0.5]{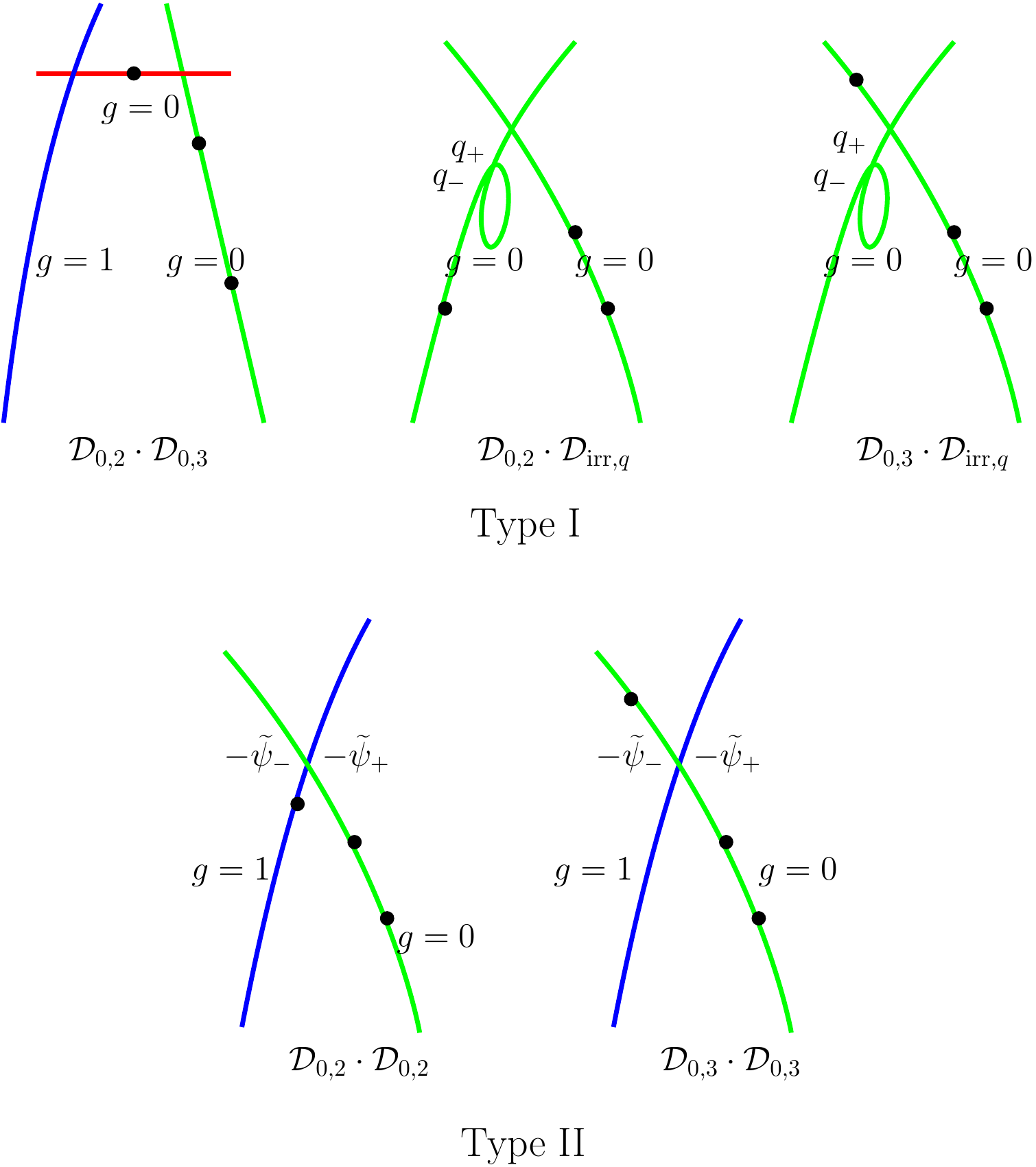}
\end{displaymath}
  \caption[doubleintersections]{Configurations in double intersections.}
  \label{figuredoubleintersections}
\end{figure}

For the self-intersections in type II, we recall the following result from \cite[Proposition 2.4.1]{FJR13}.
\begin{lemma}
\label{pullback-boundary-psi}
Let $\widetilde{\psi}_{\pm}$ be the psi-classes associated to the normalizations of the node on the two branches, denoted by $+,-$ respectively, of the partial normalization at the node.
Then 
$$\widetilde{\psi}_{\pm}={{\rm st}^*(\psi_{\pm})\over m_{+}}\,.$$
\end{lemma}
Let $\mathcal{N}_{\cD_{0,k}^{(i)}/\cW}$ be the normal sheaf of the $i$-th component $\cD_{0,k}^{(i)}$ in $\cW$, $k=2, 3$.
Lemma \ref{pullback-boundary-psi} implies
\begin{equation}
\label{self-intersection-d03}
\cD_{0,k}\cdot\cD_{0,k}
=\sum_{i, j}\delta_{i}^{j}\cD_{0,k}^{(i)}\cdot c_{1}\left(\mathcal{N}_{\cD_{0,k}^{(i)}/\cW}\right)
=\sum_{i}\cD_{0,k}^{(i)}\cdot\left(-\widetilde{\psi}_+-\widetilde{\psi}_-\right)
=\cD_{0,k}\cdot {{\rm st}^*(-\psi_{+}-\psi_{-})\over m_{+}}.
\end{equation}

The flat morphism ${\rm{st}}:\cW\rightarrow \overline{\cM}_{1,3}$ \cite[Theorem 2.2.6]{FJR13} induces a flat pullback between Chow rings  (cf. \cite{Vi89})
$${\rm{st}^*}: A^*(\overline{\cM}_{1,3})\to A^*\left(\Mbar^{1\over 3}_{1, (2^3)}\right).$$
\begin{lemma}
\label{pullback}
We have
$$\rm{st}^*\delta_{{\rm} irr}=\D_{\rm irr, 0}+6\cdot \D_{\rm irr, 1}, \quad \rm{st}^*\delta_{0,3}=3 \cD_{0,3}, \quad \rm{st}^*\delta_{0,2}=\cD_{0,2}\,.$$
\end{lemma}
\begin{proof}
We recall that each smooth point in $\Mbar_{1,3}$ has $9$ fibers, and each fiber has an automorphism group $G\cong\mu_3$, of order $3$.
For $\rm{st}^*\delta_{{\rm} irr}$, we first notice that ${\rm st}^{-1}(\delta_{\rm irr})=\amalg_{q}\, \D_{{\rm irr}, q}$ and $\D_{\rm irr, 1}=\D_{\rm irr, 2}
\in A_*(\Mbar^{1\over 3}_{1, (2^3)})$.
Following \cite[Proposition 2.2.18]{FJR13}, we analyze the degeneration of the $9$ fibers over a smooth point in $\Mbar_{1,3}$ into the singular ones in ${\rm st}^{-1}(\delta_{\rm irr})$. 
Three of the $9$ fibers degenerate to three different generic fibers in $\D_{{\rm irr}, 0}$ respectively, and all the other six fibers degenerate to a single singular fiber in $\D_{{\rm irr}, 1}$. 
Each generic fiber in each $\D_{{\rm irr}, q}$ is irreducible, thus the automorphism group of the fiber is also $G\cong\mu_3$ (\cite[Example 2.1.22]{FJR13}).
Then the first formula follows from calculation of the multiplicity.

The other two formulae are obtained similarly. We notice that the generic fiber over a point in $\delta_{0,k}$ has two irreducible components. According to \cite[Example 2.1.21]{FJR13}, 
the generic fiber over a point in $\delta_{0,3}$ has an automorphism group $G\times_{G/\langle q_+\rangle}G\cong\mu_3\times\mu_3$, of order $9$, 
while the generic fiber over a point in $\delta_{0,2}$ has an automorphism group of order $3$.
\end{proof}



\subsection{A proof of Theorem \ref{main-theorem}}
Using \eqref{eqnalgebraicdfn}, \eqref{3-spin-cubic}, \eqref{degree-equality}, and \eqref{mixed-terms}, 
we have
$$\Theta_{1,3}=
\mathrm{deg}\
{\rm st}_* 
\left(
{{\rm st}^*{\kappa_1}\over 36}-\sum\limits_{i=1}^{3}{{\rm st}^*{\psi_i}\over 36}
-{\D_{\rm irr, 0}-\D_{\rm irr, 1}-\D_{\rm irr, 2}\over 12}-{\D_{0,2}\over 12}+{\D_{0,3}\over 12}-3 \D_{0,3}^{(0)}
\right)^3\,.
$$

We split this formula into several terms and calculate each term one by one. 
Denote for simplicity
$$S:=\D_{\rm irr, 0}-\D_{\rm irr, 1}-\D_{\rm irr, 2}.$$
\begin{proposition}
\label{prop-1}
For the quadratic term and the cubic terms in $({\rm st}^*{\kappa_1}-\sum\limits_{i=1}^{3}{\rm st}^*{\psi_i})$, we have
\begin{equation*}
\label{term-12}
{\rm st}_*\left({{\rm st}^*{\kappa_1}\over 36}-\sum\limits_{i=1}^{3}{{\rm st}^*{\psi_i}\over 36}\right)^3+3\cdot{\rm st}_*\left( \big({{\rm st}^*{\kappa_1}\over 36}-\sum\limits_{i=1}^{3}{{\rm st}^*{\psi_i}\over 36}\big)^2\cdot \big(-{S\over 12}-{\D_{0,2}\over 12}+{\cD_{0,3}\over 12}-3 \D_{0,3}^{(0)}\big)\right)
={2\over 27}\cdot{1\over 12^3}\,.
\end{equation*}
\end{proposition}
\begin{proof}
Using the projection formula and the relation \eqref{g=1-kappa}, we have 
\begin{eqnarray*}
&&{\rm st}_*\left({{\rm st}^*{\kappa_1}\over 36}-\sum\limits_{i=1}^{3}{{\rm st}^*{\psi_i}\over 36}\right)^3
=\deg\cW\cdot\left({-\delta_{0,2}-\delta_{0,3}\over 36}\right)^3
=
{1\over 54}\cdot{1\over 12^3}\,.
\end{eqnarray*}
The last equality follows from \eqref{eqnstdegree} and Table \ref{table-cubic}.
Similarly, we have 
\begin{eqnarray*}
&&
3\cdot{\rm st}_*\left( \left({{\rm st}^*{\kappa_1}\over 36}-\sum\limits_{i=1}^{3}{{\rm st}^*{\psi_i}\over 36}\right)^2\cdot \left(-{S\over 12}-{\D_{0,2}\over 12}+{\cD_{0,3}\over 12}-3 \D_{0,3}^{(0)}\right)\right)\\
&=&3\cdot \big({-\delta_{0,2}-\delta_{0,3}\over 36}\big)^2\cdot \big(-{1\over 12}\cdot (1-2\cdot {1\over3})\cdot\delta_{\rm irr}-{1\over 12}\cdot3\cdot\delta_{0,2}+{1\over12}\cdot1\cdot \delta_{0,3}-3\cdot{1\over 9}\cdot \delta_{0,3}\big)\\
&=&{1\over 18}\cdot{1\over 12^3}\,.
\end{eqnarray*}
The first equality uses Corollary \ref{degree-codimension-one}.
Adding the two numbers together completes the proof.
\end{proof}

\begin{proposition}
\label{prop-2}
For the linear term in $({\rm st}^*{\kappa_1}-\sum\limits_{i=1}^{3}{\rm st}^*{\psi_i})$ without intersecting $S^2$, we have
\begin{eqnarray*}
3\cdot{\rm st}_*\left(\big({{\rm st}^*{\kappa_1}\over 36}-\sum\limits_{i=1}^{3}{{\rm st}^*{\psi_i}\over 36}\big)\cdot\big(\big(-{S\over 12}-{\D_{0,2}\over 12}+{\cD_{0,3}\over 12}-3 \D_{0,3}^{(0)}\big)^2-\big(-{S\over 12}\big)^2\big)\right)
=-{11\over 2}\cdot{1\over 12^3}\,.
\end{eqnarray*}
\end{proposition}
\begin{proof}
Using the projection formula and the relation \eqref{g=1-kappa}, the LHS of the equation above equals 
\begin{eqnarray*}
&&{1\over 12^3}\cdot
(-\delta_{0,2}-\delta_{0,3})\cdot{\rm st}_*\left(
\begin{array}{l}
2S\cdot \D_{0,2}-2S\cdot \cD_{0,3}+72S\cdot \D_{0,3}^{(0)}+\D_{0,2}\cdot\D_{0,2}-2\D_{0,2}\cdot \cD_{0,3}\\
+72\D_{0,2}\cdot \D_{0,3}^{(0)}+\cD_{0,3}\cdot \cD_{0,3}-72\cD_{0,3}\cdot \D_{0,3}^{(0)}+36^2\D_{0,3}^{(0)}\cdot \D_{0,3}^{(0)}
\end{array}
\right)\\
&=&{1\over 12^3}\left(0+{2\over 9}-8-{3\over 8}-{1\over4}+1+{1\over 72}-{1\over 9}+2\right)\\
&=&-{11\over 2}\cdot {1\over 12^3}\,.
\end{eqnarray*}
The first equality follows from the last column of Table \ref{table-2}. 

\begin{table}[h] 
\caption{Calculation for Proposition \ref{prop-2}\,.}
\label{table-2}
  \centering
  \renewcommand{\arraystretch}{1.5} 
 \begin{tabular}{|c|c|c|c|c|c|c|}
 \hline
$\Delta$ &Coeff&$|\Delta|$&$-(\delta_{0,2}+\delta_{0,3})\cdot|\Delta|$&$\deg\Delta$ &$\prod {1\over m_{+}}$&Total\\
 \hline
 $
S\cdot \D_{0,2}$&2&$\delta_{\rm irr}\delta_{0,2}$&$0$&$1-2\cdot{1\over 3}$&$1$&$0$\\
\hline
$S\cdot\D_{0,3}$&-2&$\delta_{\rm irr}\delta_{0,3}$&$-1$&${1\over 3}-2\cdot {1\over9}$&$1$&${2\over 9}$\\
\hline
$S\cdot \D_{0,3}^{(0)}$
&72&$\delta_{\rm irr}\delta_{0,3}$&$-1$&${1\over 9}-2\cdot 0$&$1$&$-8$\\
\hline
$\D_{0,2}\cdot \cD_{0,2}$&1&$\delta_{0,2}^2$&$-{1\over 8}$&$3$&${1\over 1}$&$-{3\over 8}$\\
\hline
$\D_{0,2}\cdot \D_{0,3}$&-2&$\delta_{0,2}\delta_{0,3}$&${1\over 8}$&$1$&$1$&$-{1\over4}$\\
\hline
$\D_{0,2}\cdot \D_{0,3}^{(0)}$&72&$\delta_{0,2}\delta_{0,3}$&${1\over 8}$&${1\over 9}$&$1$&$1$\\
\hline
$\D_{0,3}\cdot\D_{0,3}$&1&$\delta_{0,3}^2$&${1\over24}$&$1$&${1\over 3}$&${1\over 72}$\\
\hline
$\D_{0,3}\cdot \D_{0,3}^{(0)}$&-72&$\delta_{0,3}^2$&${1\over24}$&${1\over 9}$&${1\over 3}$&$-{1\over 9}$\\
\hline
$\D_{0,3}^{(0)}\cdot \D_{0,3}^{(0)}$ &$36^2$&$\delta_{0,3}^2$&${1\over24}$&${1\over 9}$&${1\over 3}$&$2$\\
\hline
\end{tabular}
\end{table} 

For example, the cycle $\cD_{0,3}\cdot\cD_{0,3}$ is supported on $\cD_{0,3}$. Using \eqref{self-intersection-d03},  one has
$$(-\delta_{0,2}-\delta_{0,3})\cdot {\rm st}_*\left(\cD_{0,3}\cdot\cD_{0,3}\right)
=(-\delta_{0,2}-\delta_{0,3})\cdot\delta_{0,3}^2\cdot \deg \cD_{0,3}\cdot {1\over m_+}
={1\over 24}\cdot 1\cdot{1\over 3}.$$
In Table \ref{table-2}, the intersection numbers in the fourth column follow from Table \ref{table-cubic};
$\deg \Delta$ in the fifth column
 is calculated by Corollary \ref{degree-codimension-one} and Corollary \ref{degree-codimension-two};
 and the value in the last column is the product of the values in the second, fourth, fifth, and sixth column. 
In each row, the stratum $|\Delta|$ is determined from
$$-(\delta_{0,2}+\delta_{0,3})\cdot{\rm st}_*\Delta=-(\delta_{0,2}+\delta_{0,3})\cdot \deg \Delta \cdot \prod {1\over m_+}\cdot |\Delta|\,.$$
\end{proof}

\begin{proposition}
\label{prop-3}
We have
\begin{equation*}
{1\over 12^3}\left(-3\cdot {\rm st}_*\big(S\cdot\big(-\D_{0,2}+\cD_{0,3}-36\D_{0,3}^{(0)}\big)^2\big)+{\rm st}_*\big(-\D_{0,2}+\cD_{0,3}-36\D_{0,3}^{(0)}\big)^3\right)
={389\over 54}\cdot{1\over 12^3}\,.
\end{equation*}
\end{proposition}
\begin{proof}
The explicit calculation is listed in Table \ref{table-3} below. 
In particular, the first six rows of Table \ref{table-3} gives
\begin{equation}
\label{g4-formula}
{\rm st}_*\left(S\cdot\big(-\D_{0,2}+\cD_{0,3}-36\D_{0,3}^{(0)}\big)^2\right)
=-{1\over 2}-{1\over 3}+12-{1\over 54}+{4\over 3}-24
=-{23\over 2}-{1\over 54}.
\end{equation}
The rest of Table \ref{table-3} gives
\begin{equation}
\label{g5-formula}
{\rm st}_*\big(-\D_{0,2}+\cD_{0,3}-36\D_{0,3}^{(0)}\big)^3
=-{3\over 8}+0+0+{1\over 8}-1+18+{1\over 108}-{1\over 9}+4-48
=-27-{19\over 54}\,.
\end{equation}

\begin{table}[h] 
\caption{Calculation for Proposition \ref{prop-3}\,.}
\label{table-3}
  \centering
  \renewcommand{\arraystretch}{1.5} 
 \begin{tabular}{|c|c|c|c|c|c|}
 \hline
 $\Delta$&Coeff&$|\Delta|$&$\deg\Delta$ &$\prod{1\over m_+}$&Total\\
 \hline
 $S\cdot\D_{0,2}\cdot \D_{0,2}$&$1$&$\delta_{\rm irr}\delta_{0,2}^2$=$-{3\over 2}$&$1-2\cdot{1\over 3}$&${1\over 1}$&$-{1\over 2}$\\
\hline
 $S\cdot \D_{0,2}\cdot \cD_{0,3}$&$-2$&$\delta_{\rm irr}\delta_{0,2}\delta_{0,3}$=${3\over 2}$&${1\over 3}-2\cdot {1\over 9}$&$1$&$-{1\over 3}$\\
\hline
 $S\cdot \D_{0,2}\cdot \D_{0,3}^{(0)}$&$72$&$\delta_{\rm irr}\delta_{0,2}\delta_{0,3}$=${3\over 2}$&${1\over 9}-2\cdot 0$&$1$&$12$\\
\hline
 $S\cdot \cD_{0,3}\cdot \cD_{0,3}$&$1$&$\delta_{\rm irr}\delta_{0,3}^2$=$-{1\over 2}$&$(3-2\cdot 1)\cdot{1\over 9}$&${1\over 3}$&$-{1\over 54}$\\
\hline
 $S\cdot \cD_{0,3}\cdot \D_{0,3}^{(0)}$&$-72$&$\delta_{\rm irr}\delta_{0,3}^2$=$-{1\over 2}$&$(1-2\cdot 0)\cdot{1\over 9}$&${1\over 3}$&${4\over 3}$\\
\hline
 $S\cdot \D_{0,3}^{(0)}\cdot \D_{0,3}^{(0)}$&$36^2$&$\delta_{\rm irr}\delta_{0,3}^2$=$-{1\over 2}$&$(1-2\cdot 0)\cdot{1\over 9}$&${1\over 3}$&$-24$\\
\hline
\hline
 $\D_{0,2}\cdot\D_{0,2}\cdot\D_{0,2}$&$-1$&$\delta_{0,2}^3$=${1\over 8}$&$3^2\cdot{1\over 3}\cdot{1\over 1}$&${1\over 1}$&$-{3\over 8}$\\
\hline
 $\D_{0,2}\cdot\D_{0,2}\cdot \cD_{0,3}$&$3$&$\delta_{0,2}^2\delta_{0,3}$=$0$&$3^2\cdot{1\over 3}\cdot{1\over 3\cdot 1}$&${1\over 1}$&$0$\\
\hline
 $\D_{0,2}\cdot\D_{0,2}\cdot \D_{0,3}^{(0)}$&$-108$&$\delta_{0,2}^2\delta_{0,3}$=$0$&${1\over 3}\cdot{1\over 3\cdot 1}$&${1\over 1}$&$0$\\
\hline
 $\D_{0,2}\cdot \cD_{0,3}\cdot \cD_{0,3}$&$-3$&$\delta_{0,2}\delta_{0,3}^2$=$-{1\over 8}$&$3^2\cdot {1\over 3} \cdot {1\over 3\cdot1}$& ${1\over 3}$ & ${1\over 8}$\\
\hline
 $\D_{0,2}\cdot \cD_{0,3}\cdot \D_{0,3}^{(0)}$&$6\cdot36$&$\delta_{0,2}\delta_{0,3}^2$=$-{1\over 8}$&${1\over 3}\cdot{1\over 3\cdot1}$&${1\over 3}$&$-1$\\
\hline
 $\D_{0,2}\cdot \D_{0,3}^{(0)}\cdot \D_{0,3}^{(0)}$&$-3\cdot36^2$&$\delta_{0,2}\delta_{0,3}^2$=$-{1\over 8}$&${1\over 3}\cdot{1\over 3\cdot1}$&${1\over 3}$&$18$\\
\hline
 $\cD_{0,3}\cdot \cD_{0,3}\cdot \cD_{0,3}$&$1$&$\delta_{0,3}^3$=${1\over 12}$&$3^2\cdot{1\over 3}\cdot{1\over 3}$&$({1\over3})^2$&${1\over 108}$\\
\hline
 $\cD_{0,3}\cdot \cD_{0,3}\cdot \D_{0,3}^{(0)}$&$-108$&$\delta_{0,3}^3$=${1\over 12}$&${1\over 3}\cdot{1\over 3}$&$({1\over3})^2$&$-{1\over 9}$\\
\hline
 $\cD_{0,3}\cdot \D_{0,3}^{(0)}\cdot \D_{0,3}^{(0)}$&$3\cdot36^2$&$\delta_{0,3}^3$=${1\over 12}$&${1\over 3}\cdot{1\over 3}$&$({1\over3})^2$&$4$\\
\hline
 $\D_{0,3}^{(0)}\cdot \D_{0,3}^{(0)}\cdot \D_{0,3}^{(0)}$&$-36^3$&$\delta_{0,3}^3$=${1\over 12}$&${1\over 3}\cdot{1\over 3}$&$({1\over3})^2$&$-48$\\
\hline
\end{tabular}
\end{table}

Now the result is a consequence of \eqref{g4-formula} and \eqref{g5-formula}.
\end{proof}

Finally we consider the self-intersection $S^2$ and triple-intersection $S^3$. 
By Lemma \ref{pullback}, we have
\begin{equation}
\label{self-intersection-irred}
S
=-{{\rm st}^*\delta_{\rm irr}\over 3}+{4\over 3}\cdot \cD_{\rm irr, 0}\,.
\end{equation}
According to \eqref{vanishing-irr}, $\delta_{\rm irr}^2=0$, so we apply \eqref{self-intersection-irred} and obtain
\begin{equation}
\label{double-intersection}
S^2={16\over 9}\D_{\rm irr, 0}^2-{8\over 9}\D_{\rm irr, 0}\cdot {\rm st}^*\delta_{\rm irr}\,, \quad S^3={64\over 27}\cD_{\rm irr, 0}^3-{16\over 9} \cD_{\rm irr, 0}^2\cdot {\rm st}^*\delta_{\rm irr}\,.
\end{equation}

\begin{proposition}
\label{prop-4}
We have 
\begin{equation*}
{\rm st}_*(S^3)={64\over 9}, \quad 
{\rm st}_*\left(S^2\cdot\big({\rm st}^*\kappa_1-\sum_{i=1}^3{\rm st}^*\psi_i-3\D_{0,2}+3\cD_{0,3}-108\D_{0,3}^{(0)}\big)\right)={64\over 3}.
\end{equation*}
\end{proposition}
\begin{proof}
We have 
\begin{eqnarray*}
{\rm st}_*(S^3)
&=&{\rm st}_*\left({64\over 27}\cD_{\rm irr, 0}^3-{16\over 9} \cD_{\rm irr, 0}^2\cdot {\rm st}^*\delta_{\rm irr}\right)\\
&=&{64\over 27}\cdot{\rm st}_*\left(\cD_{\rm irr, 0}\cdot (-\widetilde{\psi}_+-\widetilde{\psi}_-)^2\right)-{16\over 9}\cdot{\rm st}_*\left(\cD_{\rm irr, 0}\cdot (-\widetilde{\psi}_+-\widetilde{\psi}_-)\cdot {\rm st}^*\delta_{\rm irr}\right)\\
&=&{64\over 27}\cdot\deg\cD_{\rm irr, 0}\cdot\delta_{\rm irr}\cdot\big({-\psi_+-\psi_-\over 1}\big)^2-{16\over 9}\cdot\deg\cD_{\rm irr, 0}\cdot\delta_{\rm irr}^2\cdot\big({-\psi_+-\psi_-\over 1}\big)\\
&=&
{64\over 9}.
\end{eqnarray*}
The first equality uses the second formula in \eqref{double-intersection}.
The second equality is similar to \eqref{self-intersection-d03}.
The third equality uses the projection formula and Lemma \ref{pullback-boundary-psi}.
The psi classes $\psi_\pm$ are from the non-separating node of $\delta_{\rm irr}$.
The last equality uses Corollary \ref{degree-codimension-one} and the intersection numbers on $\overline{\mathcal{M}}_{1,3}$.

For the second formula, we observe that the projection formula and $\delta_{\rm irr}^2=0$ imply
$${\rm st}_*\left(\D_{\rm irr, 0}\cdot {\rm st}^*\delta_{\rm irr}\cdot \alpha\right)=0,  \quad \forall \alpha \in A^{*}(\cW).$$
Using this formula and the first formula in \eqref{double-intersection}, we have
\begin{eqnarray*}
&&{\rm st}_*\left(S^2\cdot\big({\rm st}^*\kappa_1-\sum_{i=1}^3{\rm st}^*\psi_i-3\D_{0,2}+3\cD_{0,3}-108\D_{0,3}^{(0)}\big)\right)\\
&=&{\rm st}_*\left({16\over 9}\D_{\rm irr, 0}\cdot \D_{\rm irr, 0}\cdot\big({\rm st}^*\kappa_1-\sum_{i=1}^3{\rm st}^*\psi_i-3\D_{0,2}+3\cD_{0,3}-108\D_{0,3}^{(0)}\big)-0\right)\\
&=&{16\over 9}\delta_{\rm irr}\cdot\big({-\psi_+-\psi_-\over 1}\big)\cdot\left(\deg\cD_{\rm irr, 0}\cdot(-\delta_{0,2}-\delta_{0,3})-3\cdot (\deg\cD_{\rm irr, 0}\cdot\cD_{0,2})\cdot\delta_{0,2}\right)+0+0\\
&=&{16\over 9}\cdot 1\cdot 3+{16\over 9}\cdot (-3)\cdot 1\cdot (-3)\\
&=&{64\over 3}.
\end{eqnarray*}
Here the $0$'s in the second equality is a consequence of the relation $\delta_{\rm irr}\cdot(-\psi_+-\psi_-)\cdot\delta_{0,3}=0.$
The third equality is a consequence of $\delta_{\rm irr}\cdot(\psi_++\psi_-)\cdot\delta_{0,2}=3.$
\end{proof}

Finally, using Proposition \ref{prop-1}, \ref{prop-2}, \ref{prop-3}, \ref{prop-4}, we obtain 
Theorem \ref{main-theorem} as follows.
\begin{proposition}
\label{prop-fjrw}
For  the LG pair $(W_3=x_1^3+x_2^3+x_3^3, \mu_3)$, the genus-one FJRW invariant 
$$\Theta_{1,3}={1\over 12^3}\left({1\over 54}+{1\over 18}-{11\over 2}+{389\over 54}-{64\over 9}+{64\over 3}\right)={1\over 12^3}\cdot 16={1\over 108}\,.$$
\end{proposition}


\section{MSP fields on cubic curves}
\label{sec:reviewMSP}

\subsection{Definition of MSP fields and basic properties}

Following \cite{CLLL1}, we define Mixed Spin P-fields for cubic curves in this part. Details can be found in \cite{CLLL1}.

Denote by 
$
\bmu_3\le \CC\sta$ the
subgroup of the $3$-th roots of unity.
Let
$$
\ti\bmu_3^+=\bmu_3\cup \{(1,\rho),(1,\varphi)\},\and 
\ti\bmu_3=\ti\bmu_3^+-\{1\}. 
$$
For $
\alpha\in\bmu_3$,   let $\langle \alpha\rangle\le\gm$ be the subgroup generated by $\alpha$; for the two exceptional elements $(1,\rho)$ and 
$
(1,\varphi)\in \ti\bmu_3^+$, we agree that
$\langle (1,\rho)\rangle=\langle (1,\varphi)\rangle=\langle 1\rangle$. 
 We pick
$$g\ge 0,\quad \gamma=(\gamma_1,\cdots,\gamma_\ell)\in 
(\ti\bmu_3)^{\times\ell},\and \bd=(d_0, d_\infty)\in
\QQ^{\times 2},
$$
and call the triple $(g,\gamma,\bd)$ a numerical type (for MSP fields).  
Throughout this work, we shall use the label 
$
m/3$ or $m$ to label the local monodromies $\gamma_{1},\cdots ,\gamma_{\ell}$ interchangably
when no confusion should arise.
For  an $\ell$-pointed twisted nodal curve $\Si^\sC\sub \sC$ over the base scheme $S$,  denote
$$\omega^{\log}_{\sC/S}:=\omega_{\sC/S}(\Si^\sC)\,.$$

\begin{definition}\label{def-curve} Let $S$ be a scheme,  $(g,\gamma,\bd)$ be a numerical type.
An $S$-family of MSP-fields   of type $(g,\gamma,\bd)$ is a datum
$$
\xi=(\sC,\Si^\sC,\sL,\sN, \varphi,\rho,\nu)
$$
such that
\begin{enumerate}
\item[(1)] $\cup_{i=1}^\ell\Si_i^\sC = \Si^\sC\subset \sC$ is an $\ell$-pointed, genus $g$,
twisted curve over $S$ such that the $i$-th marking $\Si^\sC_i$ is banded by the group $\langle\gamma_i\rangle\le \gm$;
\item[(2)] $\sL$ and $\sN$ are representable invertible sheaves on $\sC$,  $\sL\otimes \sN$ and $\sN$ have fiberwise degrees $d_0$ and $d_\infty$ respectively. The monodromy of $\sL$ along $\Si^\sC_i$ is
$\gamma_i$ when $\langle \gamma_i\rangle\ne\langle 1\rangle$;
\item[(3)] $\nu=(\nu_1, \nu_2)\in H^0( \sL\otimes\sN)\oplus  H^0( \sN)$ such that $(\nu_1,\nu_2)$ is nowhere vanishing;
\item[(4)]
$\varphi=(\varphi_1,\varphi_2, \varphi_{3}) \in H^0(\sL^{\oplus 3})$, $(\varphi,\nu_1)$ nowhere zero,
and $\varphi|_{\Si^\sC_{(1,\varphi)}}=0$;
\item[(5)] $\rho \in H^0(\sL^{- 3}\otimes \omega^{\log}_{\sC/S})$;
$(\rho,\nu_2)$ is nowhere vanishing, and $\rho|_{\Si^\sC_{(1,\rho)}}=0$.
\end{enumerate}
\end{definition}

We define $\cW\lggd^{\pre}$ to be the category fibered in groupoids over
the category of schemes, such that objects in $\cW\lggd^{\pre}(S)$ are
$S$-families of MSP-fields, and morphisms are naturally defined as in \cite{CLLL1}.

\begin{definition} We call $\xi\in \cW\lggd^{\pre}(\CC)$ {\em stable} if $\Aut(\xi)$ is finite.
We call $\xi\in \cW\lggd^{\pre}(S)$ stable if $\xi|_s$ is stable for every closed point $s\in S$.
\end{definition}

Let $\cW\lggd\subset \cW^{\pre}\lggd$ be the open substack of families of stable objects in $\cW\lggd^{\pre}$.
We introduce a $T=\gm$ action on $\cW\lggd$ by
\begin{eqnarray}\label{Gm}
t\cdot (\Si^\sC, \sC, \sL, \sN,\varphi,\rho, (\nu_1, \nu_2))
=  (\Si^\sC, \sC, \sL, \sN,\varphi,\rho, (t\nu_1,\nu_2)),\quad t\in \gm.
\end{eqnarray}
Using the same method in \cite{CLLL1}, one can prove that the stack $\cW\lggd$ is a DM $T$-stack, locally of finite type. 

The polynomial $W=x_1^3+x_2^3+x_3^3$ gives a $T$-equivariant cosection $\sigma$ of the obstruction sheaf $\Ob_{\cWgg}$. The degeneracy locus of $\sigma$ given by
\begin{eqnarray*}
\cWgg^{-}(\CC)=\{\xi\in \cWgg(\CC)\mid \sigma|_\xi=0\}
\end{eqnarray*}
is a proper substack of $\cWgg$. Therefore, the moduli stack $\cW\lggd$ admits a cosection localized virtual cycle $[\cW\lggd]^{\rm vir}_{\rm loc}\subset \cWgg^{-}(\CC)$
as described in \cite{CLL}.

Let $0\le m_i\le 2$ be so that $\gamma_i=\zeta_3^{m_i}$, where $\zeta_3=e^{\frac{2\pi\sqrt{-1}}{3}}$.
 Let $\ell_\varphi$ 
be the number of $\gamma_i$ so that $\gamma_i=(1,\varphi)$; likewise for $\ell_\rho$. We let
$\ell_0=\ell-\ell_\varphi-\ell_\rho$.

The virtual dimension of the moduli stack $\cW\lggd$ is (see \cite{CLLL1} for the calculation) 
\begin{align}\label{vdim}
\delta(g,\gamma,\bd)
\colon=
 d_0+d_\infty+g-1+\ell -2\left(\ell_\varphi+\sum_{i=1}^{\ell_0}\frac{m_i}{3}\right).
\end{align}

\subsection{Localization formulae}\label{localization-factors}
In this part, we compute all the quantities for torus localization. 
The computations are the same as those in \cite{CLLL2}, hence we shall only list the results and skip the details. \\

For simplicity, we denote $\cW:=\cW\lggd$.
For each $\xi$ in the fixed locus $\cW^T$, one can associate a decorated graph $\Gamma_\xi$.
Before defining this graph $\Gamma_\xi$, we introduce a decomposition of the curve $\sC$ as follows. 
\begin{definition} \label{Ca}
Given $\xi\in \fixW$,  let
$\sC_0=\sC\cap (\nu_1=0)_\redd$,  $\sC_\infty=\sC\cap (\nu_2=0)_\redd$,
$\sC_1=\sC\cap (\rho=\varphi=0)_\redd$,
 $\sC_{01}$ (resp. $\sC_{1\infty}$) be the union of irreducible components of
$\overline{\sC-\sC_0\cup\sC_1\cup\sC_\infty}$ in
$(\rho=0)$ (resp. in $(\varphi=0)$), and  $\sC_{0\infty}$ be the union of irreducible components of
$\sC$ not contained in $\sC_0\cup\sC_1\cup\sC_\infty\cup\sC_{01}\cup\sC_{1\infty}$.
\end{definition}

\begin{definition}\label{graph1}
To each $\xi\in \cW^T$ we associate  a graph $\Ga_\xi$ as follows:

\begin{enumerate}
\item (vertex) let $V_0(\Ga_\xi)$, $V_1(\Ga_\xi)$, and $V_\infty(\Ga_\xi)$ be
the set of connected components of $\sC_0$, $\sC_1$, $\sC_\infty$ respectively, and
let $V(\Ga_\xi)$ be their union; 

\item (edge)   let $E_0(\Ga_\xi)$, $E_\infty(\Ga_\xi)$ and $E_{0\infty}(\Ga_\xi)$
be the set of irreducible components
of $\sC_{01}$, $\sC_{1\infty}$ and $\sC_{0\infty}$ respectively, and  let $E(\Ga_\xi)$
be their union; 

\item (leg) let $L(\Ga_\xi)\cong \{1,\cdots,\ell\}$ be the ordered set of markings of $\Si^\sC$, $\Si_i^\sC\in L(\Ga_\xi)$ is attached to
$v\in V(\Ga_\xi)$ if $\Si_i^\sC\in \sC_v$;

\item (flag) $(e,v)\in F(\Ga_\xi)$ if and only if $\sC_e\cap \sC_v\ne \emptyset$. 

\end{enumerate}
Here, $\sC_a$ is  the curve associated to the symbol $a\in V(\Ga_\xi)\cup E(\Ga_\xi)$.
We call $v\in V(\Ga_\xi)$ stable if $\sC_v\sub\sC$ is one-dimensional, otherwise we call it  unstable.
\end{definition}
Let $V^S(\Ga_\xi)\subset V(\Ga_\xi)$ be the set of stable vertices and $V^U(\Ga_\xi)\subset V(\Ga_\xi)$ be
the set  of unstable vertices.
Given $v\in V(\Ga_\xi)$,  let $E_v=\{e\in E(\Ga_\xi): (e,v)\in F(\Ga_\xi)\}$.
 We denote
\begin{eqnarray}\label{VV0}
V^{a,b} (\Ga_\xi)=\{v\in V(\Ga_\xi)-V^S(\Ga_\xi)\, |\,  |S_v|=a, |E_v|=b\}.
\end{eqnarray}
We also abbreviate $V_c^{a,b}(\Ga_\xi)=V_c(\Ga_\xi)\cap V^{a,b}(\Ga_\xi)$ for $c\in \{0,1,\infty\}$.\\

  Let $\Gamma$ be a regular graph (see the definition of regular graph in \cite{CLLL2}) associated to $\xi \in \cW$.  Let
$\Wfix$ be the groupoid of $\Ga$-framed families in $\fixW$ with obviously defined arrows. 
  We define 
  \begin{eqnarray}\label{VV}
\sV=(\sL(-\Si^{\sC}\lophi))^{\oplus 3}\oplus \sL^{\vee\otimes 3}\otimes\omega_{\sC}^{\log}(-\Si\lorho)\oplus \sL\otimes\sN\otimes\bL_1\oplus\sN\,,
\end{eqnarray}
where   $\bL_k$ is  the one-dimensional weight-$k$ $T$-representation.

We denote
\begin{eqnarray*}
B_1 &=& \Aut(\Si^{\sC}\subset {\sC})= \Ext^0(\Omega_{\sC}(\Si^{\sC}),\cO_{\sC}),\\
B_2 &=& \Aut({\sL}) \oplus \Aut({\sN})= H^0(\cO_{\sC}^{\oplus 2}),\\ 
B_3 &=& \Def(\varphi,\rho, (\nu_1,\nu_2)) = H^0({\sV}), \\
B_4 &=& \Def(\Si^{\sC}\subset {\sC})= \Ext^1(\Omega_{\sC}(\Si^{\sC}),\cO_{\sC}),\\
B_5 &=& \Def({\sL})\oplus \Def({\sN})= H^1(\cO_{\sC}^{\oplus 2}), \\
B_6 &=& \Obs(\varphi,\rho, (\nu_1,\nu_2)) = H^1({\sV}), 
\end{eqnarray*}
where $\Aut$ is the space of infinitesimal automorphisms, 
$\Def$ is the space of infinitesimal deformations,
and $\Obs$ is the obstruction space.
Then, all $B_i$ are $T$-spaces. Let $B_i\umv$ be the moving part of $B_i$. Then, 
the virtual normal bundle $N^\vir$ to $\cW_\Ga$ in $\cW\lggd$ restricted at $\xi$ is 
$$
N^\vir |_\xi= T_\xi\umv-Ob_\xi\umv = -B_1\umv - B_2\umv + B_3\umv + B_4\umv + B_5\umv -B_6\umv.
$$

We introduce the following  convention on nodes:
\begin{equation}
\forall\, (e,v)\in F:\ y_{(e,v)}={\sC}_v\cap {\sC}_e;\quad \forall\, v\in V^{0,2}\ \text{and}\ E_v=\{e,e'\}:\ y_{(e,v)}=\sC_e\cap\sC_{e'}\,.
\end{equation}
Let $F^{0,1}=\{(e,v)\in F: v\in V^{0,1}\}$ and $F^S=\{(e,v)\in F:v\in V^S\}$. Then we can obtain the following 
\begin{equation}\label{eqn:B-one-four}
\frac{e_T(B_1\umv)}{e_T(B_4\umv)} =  
\prod_{(e,v)\in F^S}\frac{1}{w_{(e,v)}-\psi_{(e,v)}} 
\cdot  \prod_{\substack{v\in V^{0,2}\\ E_v=\{e,e'\} }}\frac{1}{w_{(e,v)} + w_{(e',v)}}
\cdot  \prod_{(e,v)\in F^{0,1}}\! w_{(e,v)},
\end{equation}
 where $\psi_{(e,v)} $ is the $\psi$-class associated to the pointed curves $y_{(e,v)}\in \sC_v$ and $w_{(e, v)}$ will be given in Lemma \ref{wev} below.

For $e\in E_0\cup E_\infty$,  let $d_e=\deg(\sL|_{\sC_e})$. For $e\in E_\infty$,  let
\begin{equation}\label{eqn:re}
r_e=\begin{cases}
1, & d_e \in \ZZ,\\
3, & d_e\notin \ZZ.
\end{cases}
\end{equation}
Then $\sC_e\cong \PP(r_e,1)$. Since $\Gamma$ is regular,
if $v\in V^S_\infty$, then $d_e\notin \ZZ$ and $r_e=3$. Note that
if $v\in V^{0,1}_\infty$, then $d_e\in \ZZ$ and $r_e=1$.

\begin{lemma}\label{wev}
\begin{enumerate}
\item When $v\in V_0$, then $w_{(e,v)}=\frac{h_e +\ft}{d_e}$ and $ w_{(e,v')}=-\frac{h_e+\ft}{d_e}$.
\item When $v\in V_\infty\setminus V_\infty^{0,1}$, then $w_{(e,v)}=\frac{\ft}{r_e d_e}$ and $ w_{(e,v')}=-\frac{\ft}{d_e}$.
\item When $v\in V_\infty^{0,1}$, then  $w_{(e,v)}=\frac{3\ft}{3d_e+1}$ and $ w_{(e,v')}=\frac{-3\ft}{ 3d_e+1}$.
\end{enumerate}
Here $\ft$ is the equivariant parameter in  $H^2(BT)$.
\end{lemma}

\subsubsection{Contribution from stable vertices} \label{sec:vertex}
We introduce some more notations:
\begin{itemize}
\item Given a stable vertex $v\in V^S$, let $\pi_v: \cC_v\to \cW_v$ be the universal curve associated to the regular graph $\Gamma=\{v\}$; let $\cL_v$ and $\cN_v$ be
the universal line bundle over $\cC_v$.
\item Given $v\in V^S_0$, let $\phi_v: \cC_v\to \PP^3$ be induced by the sections in $\varphi$.
\item Given $v\in V^S_1$, let $\EE:= \pi_*\omega_{\pi_v}$ be the Hodge bundle, where
$\omega_{\pi_v}\to \cC_v$ is the relative dualizing sheaf. Then $\EE^\vee=R^1\pi_{v*}\cO_{\cC_v}$.
\end{itemize}
The contribution $A'_v$ from a stable vertex $v\in V^S$ is given by 
the following lemma.
\begin{lemma}\label{lem:VS}
$$
A'_v=\begin{cases}
\displaystyle{ \frac{1}{e_T(R\pi_{v*} \phi_v^* \cO_{\PP^4}(1)\otimes \bL_1)} },
& v\in V^S_0; \\
& \\
\displaystyle{\Big(\frac{e_T(\EE_v^\vee \otimes \bL_{-1})}{-\ft}\Big)^3
	\cdot \frac{3\ft}{e_T(\EE_v\otimes \bL_3)}
	\cdot\left(\frac{1}{3\ft}\right)^{|E_v|}\left(\frac{-\ft^2}{3}\right)^{|S_v^{(1,\varphi)}|} , } 
& v\in V_1^S;\\
\displaystyle{ \frac{1}{e_T(R\pi_{v*}\cL_v^\vee \otimes \bL_{-1})} }, & v\in V_\infty^S.
\end{cases}
$$

\end{lemma}

\subsubsection{Contribution from edges}\label{sec:edge}
In this part, we compute the contribution  from  edges $e\in E(\Ga)$.  

\begin{lemma}\label{lem:all-edges}
The contribution from edges is given by 
	$$
	 \prod_{e\in E} A_e \prod_{v\in V_1^{0,1}\cup V_1^{1,1}} A_v,
	$$
	where $A_e$ and $A_v$ are defined as follows.
	$$
	A_e = \begin{cases}
	\displaystyle{ \frac{\prod_{j=1}^{3d_e-1}(-3 h_e +\frac{j(h_e+\ft)}{d_e})}{
			\prod_{j=1}^{d_e}(h_e-\frac{j(h_e+\ft)}{d_e})^3 \cdot \prod_{j=1}^{d_e}\frac{j(h_e+\ft)}{d_e} }}, 
	& e\in E_0; \\
	&\\
	\displaystyle{\frac{\prod_{j=1}^{\lceil-d_e\rceil-1}(-\ft-\frac{j\ft}{d_e} )^3 }{
			\prod_{j=1}^{{-3d_e}}(-\frac{j\ft}{d_e}) 
			\prod_{j=1}^{\lfloor-d_e \rfloor}(\frac{j\ft}{d_e})}
	},& e\in E_\infty,\ (e,v)\in F,\ v\in V_\infty \setminus V_\infty^{0,1};\\
	& \\
	\displaystyle{\frac{\prod_{j=1}^{-d_e{-1}}(-\ft-\frac{3j\ft}{3d_e+1} )^3 }{
			\prod_{j=1}^{{ -3d_e-1}}(\frac{-3j\ft}{3d_e+1}) 
			\prod_{j=1}^{-d_e}(\frac{3j\ft}{3d_e+1})}
	},& e\in E_\infty,\ (e,v)\in F, \ v\in V_\infty^{0,1}.
	\end{cases} 
	$$

 	$$
	A_v= \begin{cases}
   	3\ft, & v\in V_1^{0,1};\\
 	-\ft^3, & v\in V_1^{1,1},\  S_v\subset \Si^{\sC}\lophi;\\
 	3t, & v\in V_1^{1,1},\ S_v\subset \Si^{\sC}\lorho. 
	\end{cases} 
	$$ 
\end{lemma}

\subsubsection{Contributions from nodes} \label{sec:node}

The following results are straightforward from \cite{CLLL2}.
\begin{lemma}[Contribution from a flag in  $F^S$]\label{lem:FS}
	If $(e,v)\in F^S$ then
	$$
	A_{(e,v)}=\begin{cases}
	h_e+\ft, & v\in V_0;\\
	-3\ft^4, & v\in V_1;\\
	1, & v\in V_\infty.
	\end{cases}
	$$
\end{lemma}

\begin{lemma}[Contribution from a vertex in $V^{0,2}$] \label{lem:V-zero-two}
	If $v\in V^{0,2}$, then
	$$
	A_v= \begin{cases}
	h_e +\ft = h_{e'}+\ft, & v\in V_0^{0,2}\ \text{and}\ E_v=\{e,e'\};\\
	-3\ft^4, & v\in V_1^{0,2};\\
	(-\ft)^{{6}\epsilon(d_e)} , & v\in V_\infty^{0,2}\ \text{and}\ E_v=\{e,e'\}.\\
	\end{cases} 
	$$
	Here we define $\ep(x)=1$ when $x\in \ZZ$, and $\ep(x)=0$ otherwise.
\end{lemma} 

\subsubsection{Summary}\label{sec:BBBBsummary}
For $v\in V^S$, let $A_v'$ be as in Lemma \ref{lem:VS}, and define
\begin{equation}\label{eqn:Av}
	\begin{aligned}
		& A_v := A_v'\prod_{e\in E_v} A_{(e,v)}\\
		& = \begin{cases}
			\displaystyle{ \frac{1}{e_T(R\pi_{v*} {\phi}_v^* \cO_{\PP^2}(1)\otimes \bL_1)}
				\prod_{e\in E_v}(h_e+\ft) },
			& v\in V^S_0; \\
			\displaystyle{\Big(\frac{e_T(\EE_v^\vee \otimes \bL_{-1})}{-\ft}\Big)^3
				\cdot \frac{3\ft}{e_T(\EE_v\otimes \bL_3)}\cdot (-\ft^3)^{|E_v|}\cdot \left(\frac{-\ft^2}{3}\right)^{|S_v^{(1,\varphi)}|}  , } 
			& v\in V_1^S;\\
			\displaystyle{ \frac{1}{e_T(R\pi_{v*}\cL_v^\vee \otimes \bL_{-1})} 
				 }, & v\in V^S_\infty.
		\end{cases}
	\end{aligned}
\end{equation}
Then
\begin{equation}\label{eqn:AvAev}
	\prod_{v\in V^S}A_v =\prod_{v\in V^S}A_v'\prod_{(e,v)\in F^S} A_{(e,v)}.
\end{equation}

\begin{proposition}\label{prop:BBBB}
	Define $A_v$ by \eqref{eqn:Av} if $v\in V^S$. Let 
	$A_v$ be given by  Lemma \ref{lem:all-edges} (resp. Lemma \ref{lem:V-zero-two}) if $v\in V_1^{0,1}\cup V_1^{1,1}$ (resp. $v\in V^{0,2}$). If  
	$v\in V_0^{0,1}\cup V_0^{1,1}\cup V_\infty^{0,1}\cup V_\infty^{1,1}$, define $A_v=1$.
	For $e\in E$, let $A_e$ be as in Lemma \ref{lem:all-edges}. Then
	$$
	\frac{e_T(B_2\umv)e_T(B_6\umv)}{e_T(B_3\umv)e_{T}(B_5\umv)} 
	=\prod_{v\in V} A_v \prod_{e\in E} A_e.
	$$
\end{proposition}

\subsection{Equivariant Euler class of the virtual normal bundle} \label{sec:formula}
Let $A_v$ and $A_e$ be defined as in Proposition \ref{prop:BBBB}. Combining 
\eqref{eqn:B-one-four}  and Proposition \ref{prop:BBBB}, we obtain
\begin{theorem}\label{final}
	$$
	\frac{1}{e_T(N^\vir_{\vGa})} =\prod_{v\in V(\Ga)}B_v \prod_{e\in E(\Ga)} A_e,
	$$
	where 
	$$
	B_v=\begin{cases}
	A_v\cdot \displaystyle{ \prod_{e\in E_v}\frac{1}{w_{(e,v)}-\psi_{(e,v)}} } , &  v\in V^S;\\
	\displaystyle{ \frac{A_v}{w_{(e,v)}+w_{(e,v')}} }, & v\in V^{0,2}, E_v=\{e,e'\};\\
	A_v, & v\in V^{1,1};\\
	A_v\cdot w_{(e,v)}, & v\in V^{0,1}, E_v=\{e\}.
	\end{cases}
	$$
\end{theorem}

The contribution of each graph $\Gamma$ follows from a localization formula
$${\rm Cont}(\Gamma)=\left[t^{\delta(g, {\bf d})}\cdot{[\mathcal{W}_{(\Gamma)}]^{\rm vir}_{\rm loc} \over e(N_{\Gamma})}\right]_0\,.$$
Similar to \cite[Proposition 3.5]{CLLL2}, we have
\begin{equation}
\label{msp-virtual-cycle}
[\mathcal{W}_{(\Gamma)}]^{\rm vir}_{\rm loc}
={1\over {\rm Aut}(\Gamma)}
\cdot{1\over \prod\limits_{e\in E_{\infty}}d_e}
\cdot{1\over \prod\limits_{e\in E_{\infty}}|G_e|}
\cdot
(\iota_{\Gamma}^-)_*\left(\prod_{v}[\mathcal{W}_{v}]^{\rm vir}_{\rm loc}\right)
\end{equation}
Set $\delta=-1$ if $v\in V_{\infty}^{0,1}$ and $\delta=0$ otherwise.
Similar to \cite[Lemma 2.31]{CLLL2}, we have
\begin{equation}\label{value-Ge}
|G_e|=|-3d_e+\delta|.
\end{equation}

\section{A genus-one FJRW invariant via MSP fields}
\label{sec:MSPlocalization}

Now we use the Mixed Spin P-fields method described in the previous section to prove
\begin{proposition}
\label{prop-msp}
We have the following primitive FJRW invariant
\begin{equation}
\label{invariant}
\Theta_{1,3}=\deg\left(\left[\overline{\M}_{1,(2^3)}(W_3, \mu_3)^p\right]^{\rm vir }\right)
={1\over 108}.
\end{equation}
\end{proposition}

\subsection{Regular graphs in MSP fields}
We classify all regular graphs for the case $$(g, \ell, d_0, d_\infty)=(1, 0, 0, 1).$$
There are eight graphs in total. 
One of the graphs is  $\Gamma_{0}$ 
as depicted in Figure \ref{figureGamma0} below.
\begin{figure}[H]
  \renewcommand{\arraystretch}{1} 
\begin{displaymath}
\includegraphics[scale=0.45]{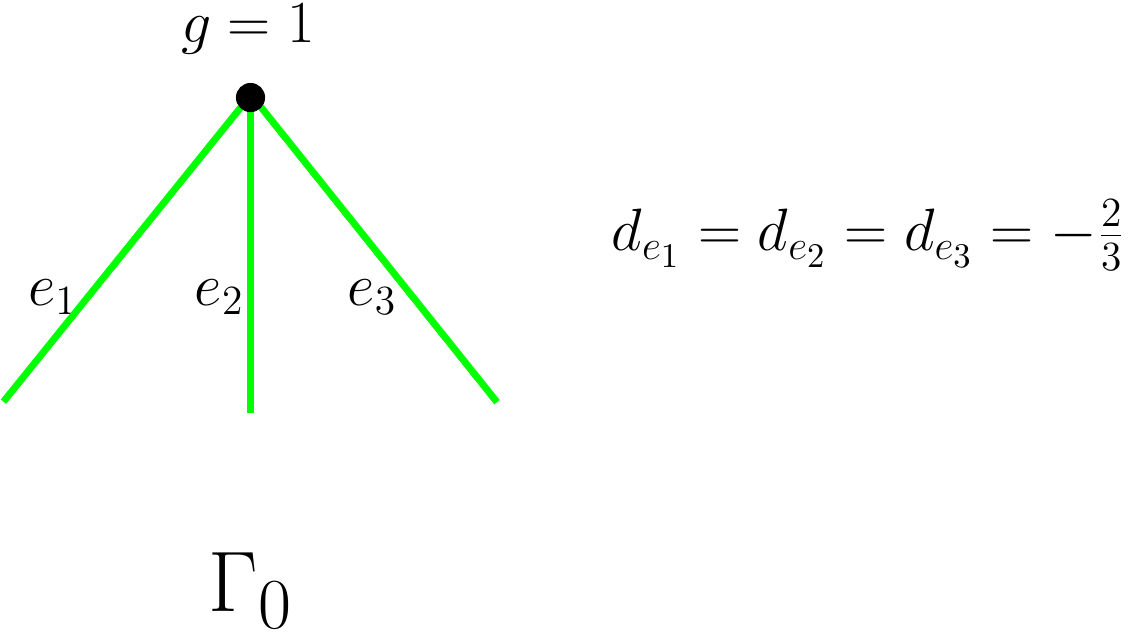}
\end{displaymath}
  \caption[Gamma0]{Configuration represented by the graph $\Gamma_{0}$.}
  \label{figureGamma0}
\end{figure}
In such a graph, we label the vertices from top to bottom, and from left to right. We label the edges from left to right, and decorate each edge by its degree.
For stable vertices, we also decorate the vertex with its genus.
All three edges in the graph $\Gamma_{0}$ has $d_e=-2/3$, thus the genus one vertex $v\in V_{\infty}^S$ has a $\frac{2}{3}$-insertion at each of the three nodes. The vertex $v\in V_{\infty}^S$ contributes as some dual-twisted FJRW invariant, and can be simplified to the primitive FJRW invariant $\Theta_{1,3}$ in \eqref{invariant}.

We list all other regular graphs in Figure \ref{figureGammagraphs} below:
\begin{figure}[H]
  \renewcommand{\arraystretch}{1} 
\begin{displaymath}
\includegraphics[scale=0.45]{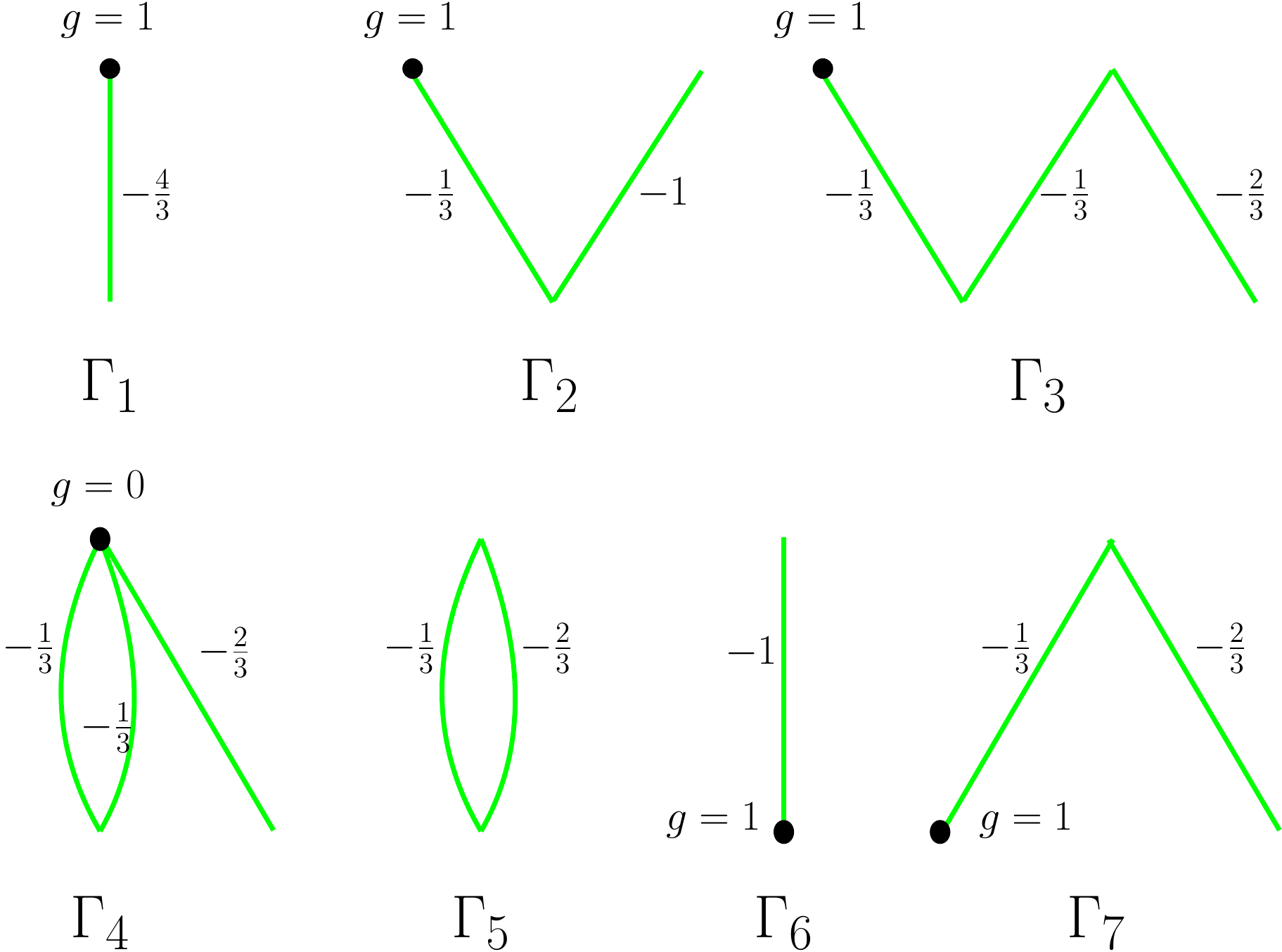}
\end{displaymath}
  \caption[Gammagraphs]{Configurations represented by the graphs $\Gamma_{1},\cdots ,\Gamma_{7}$.}
  \label{figureGammagraphs}
\end{figure}

All the contributions ${\rm Cont}(\Gamma_i)$ ($i=0, \cdots, 7$) will be calculated explicitly in what follows.



\subsection{Explicit computations}
\label{explicit-MSP}
By the formula \eqref{vdim}, the virtual dimension of the relevant moduli stack is $1$. 
Using  MSP fields and localization in the previous section, 
we have
\begin{equation}
\label{msp-graph-sum}
\sum_{i=0}^{7}{\rm Cont}(\Gamma_i)=0.
\end{equation}


Let us first consider the contribution from the graph $\Gamma_0$.
\begin{lemma}
Let $\Theta_{1,3}$ be the primitive FJRW invariant defined in \eqref{invariant},
we have 
\begin{equation}
\label{msp-fjrw}{\rm Cont}(\Gamma_0)=-{1\over 6}\cdot\Theta_{1,3}.
\end{equation}
\end{lemma}
\begin{proof}
According to our convention,
$$v_1\in V_{\infty}^{S}, \quad v_2, v_3, v_4\in V_{1}^{0,1}, \quad e_1, e_2, e_3 \in E_{\infty}.$$
In this case, the genus of $v_1$ is $g(v_1)=1$, and
$$\mathscr{L}_{v_1}=\mathcal{O}_\sC(-1)\left({2\over 3}p_1+{2\over 3}p_2+{2\over 3}p_3\right)\,, 
\quad \mathscr{L}_{v_1}^{\vee}=\mathcal{O}_\sC(1-3)\left({1\over 3}p_1+{1\over 3}p_2+{1\over 3}p_3\right)\,.$$
  Using  orbifold Riemann-Roch formula \cite{AGV}, we have  
${\rm rank} (R\pi_{v_1*}\sL_{v_1}^\vee)=-2.$
The virtual dimension equals the dimension of the moduli space and
 only $c_0(R\pi_*\sL^\vee)$ contributes.
Similar to \cite[Proposition 3.5]{CLLL2}, 
for $\Gamma_0$, we have
\begin{eqnarray*}
\prod_{i=2}^{4}B_{v_i}&=&\prod_{i=2}^{4}A_{v_i} w_{(e_{i-1}, v_i)}
=\left(3\ft(-{\ft\over -{2\over 3}})\right)^3
={729\over 8}\ft^6\,,\\
\prod_{i=1}^{3}A_{e_i}&=&\left({1\over \prod_{j=1}^2(-{j \ft\over -{2\over 3}})}\right)^3
={8\over 729}\ft^{-6}\,.
\end{eqnarray*}
We obtain 
\begin{eqnarray}
{\rm Cont}(\Gamma_0)
&=&
\left[\ft^{\delta}\cdot{1\over 3! 2^3}\cdot
{729\ft^6\over 8}\cdot{8\over 729\ft^6}\cdot
{[\overline{\mathcal{M}}_{g,\gamma}(G)^p]^{\rm vir}_{\rm loc}\over e_{T}(R\pi_{v_1*}\sL_{v_1}^\vee\otimes{\bf L}_{-1})}\cdot
\prod_{i=1}^3{1\over -{\ft\over 2}-\psi_i}\right]_0\nonumber\\
&=&\left[\ft^{1}\cdot{1\over 3! 2^3}\cdot
{[\overline{\mathcal{M}}_{g,\gamma}(G)^p]^{\rm vir}_{\rm loc}\over 
(-\ft)^{-2}}\cdot
\prod_{i=1}^3{1\over -{\ft\over 2}-\psi_i}\right]_0\nonumber\\
&=&
-{1\over 6}\cdot\Theta_{1,3}.\nonumber
\end{eqnarray}
\end{proof}


Now let us consider the contribution from the other graphs. We need a preparational lemma.
\begin{lemma}\label{lem:psi-vanishing}
We have the integral
$$\int_{[\overline{\mathcal{M}}_{1,(1)}^{1/3, p}]^{\rm vir}_{\rm loc}}\psi_1=0. 
$$
\end{lemma}
\begin{proof}
Similar to \cite[Lemma 5.7]{CLLL2}, we have
$$[\overline{\mathcal{M}}_{1,(1)}^{1/3, p}]^{\rm vir}_{\rm loc}=-2^3[M_0]+[M_1],$$
where $M_0$ is the locus where $\sL(-\Si^{\sC})\cong \cO_{\sC}$, and $M_1$ is the locus where $\sL(-\Si^{\sC})\ncong \cO_{\sC}$ (i.e., $\sL(-\Si^{\sC})$ has no section), which have $3^2-1$ components.
We have
\begin{equation}
\label{psi-vanishing}
\int_{[\overline{\mathcal{M}}_{1,(1)}^{1/3, p}]^{\rm vir}_{\rm loc}}\psi_1=
-2^3\int_{[M_0]}\psi_1+\int_{[M_1]}\psi_1=-2^3\cdot{1\over 3\times 3\times 24}+{3^2-1\over 24\times 3\times 3}=0.
\end{equation}
\end{proof}

\begin{lemma}
\label{lemma-graph123}
The total contribution from $\Gamma_1$, $\Gamma_2$, and $\Gamma_3$ vanishes. That is,
$${\rm Cont}(\Gamma_1)+{\rm Cont}(\Gamma_2)+{\rm Cont}(\Gamma_3)=0.$$
\end{lemma}
\begin{proof}
We have (see the details in Appendix \ref{appendix-msp})
\begin{equation}
\label{first-three-graphs}
{\rm Cont}(\Gamma_1)={1 \over 648}\cdot [B_{v_1}]_{\ft^{-1}},
\quad
{\rm Cont}(\Gamma_2)
=-{2\over 81}\cdot [B_{v_1}]_{\ft^{-1}},
\quad
{\rm Cont}(\Gamma_3)
={3\over 162}\cdot [B_{v_1}]_{\ft^{-1}}.
\end{equation}
Here, each $[B_{v_1}]_{\ft^{-1}}$ depends on the degree $d_e$, where $e\in E_{\infty}$ is the edge with $v_1$ one of its vertex. 
Let us compute the term $[B_{v_1}]_{\ft^{-1}}$ for each graph.
We have 
$$\mathscr{L}_{v_1}=\mathcal{O}_{\sC}\left({1\over 3}p_1\right)\,,\quad \mathscr{L}_{v_1}^\vee=\mathcal{O}_{\sC}(-1)\left({2\over 3}p_1\right).$$
We have $g(v_1)=1$ and ${\rm rank}(R\pi_{v_1*}\mathscr{L}_{v_1}^\vee)=-1$. The total Chern class is multiplicative, thus
$${1\over c(R\pi_{v_1*}\mathcal{L}_{v_1}^\vee\otimes{\bf L}_{-1})}=c(-R\pi_{v_1*}\mathscr{L}_{v_1}^\vee\otimes{\bf L}_{-1})
=1+c_1(-R\pi_{v_1*}\mathcal{L}_{v_1}^\vee\otimes{\bf L}_{-1}).$$
Therefore 
\begin{eqnarray*}
e^{-1}_{T}(R\pi_{v_1*}\mathscr{L}_{v_1}^\vee\otimes{\bf L}_{-1})
=
c_1(-R\pi_{v_1*}\mathscr{L}_{v_1}^\vee\otimes{\bf L}_{-1})
=-\ft-{\rm ch}_1(R\pi_{v_1*}\mathscr{L}_{v_1}^\vee)\,.
\end{eqnarray*}
Now we have
\begin{eqnarray*}
[B_v]_{\ft^{-1}}&=&\left[{1\over  e_{T}(R\pi_{v_1*}\mathcal{L}_{v_1}^\vee\otimes{\bf L}_{-1})}
\prod_{e\in E_v}{1\over w_{(e,v)}-\psi_{(e,v)}}\right]_{\ft^{-1}}\\
&=&\left[{(-\ft)-{\rm ch}_1(R\pi_{v_1*}\mathscr{L}_{v_1}^\vee)\over {\ft\over3d_{e_1}}-\psi_{(e_1, v_1)}}
\right]_{\ft^{-1}}\\
&=&
-3d_{e_1}({\rm ch}_1(R\pi_{v_1*}\mathscr{L}_{v_1}^\vee)+3d_{e_1}\psi_{(e_1, v_1)}).
\end{eqnarray*}
Since $d_{e_1}= -{4\over 3}, -{1\over 3}$, and $-{1\over 3}$ in $\Gamma_1, \Gamma_2$, and $\Gamma_3$ respectively, using \eqref{first-three-graphs} and Lemma \ref{lem:psi-vanishing}, we obtain
$${\rm Cont}(\Gamma_1)+{\rm Cont}(\Gamma_2)+{\rm Cont}(\Gamma_3)
=\int_{[\overline{\mathcal{M}}_{1,(1)}^{1/3, p}]^{\rm vir}_{\rm loc}}\left(0\cdot {\rm ch}_1(R\pi_{v_1*}\mathscr{L}_{v_1}^\vee)-{1\over 54}\cdot\psi_1\right)=0\,.$$
\end{proof}
\begin{lemma}
\label{lemma-graph4}
For the contribution of the graph $\Gamma_4$, we have
$${\rm Cont}(\Gamma_4)=-{1\over 108}\,.$$
\end{lemma}
\begin{proof}
The automorphism of $\Gamma_4$ is nontrivial: $|{\rm Aut}(\Gamma_4)|=2.$
The vertex $v_1\in V^S$ is   stable,  and
$$\mathscr{L}_{v_1}=\mathcal{O}(-1)\left({1\over 3}p_1+{1\over 3}p_2+{2\over 3}p_3\right), \quad \mathscr{L}_{v_1}^\vee=\mathcal{O}(-2)\left({2\over 3}p_1+ {2\over 3}p_2+ {1\over 3}p_3\right).$$
 We have $g(v_1)=0$, ${\rm rank} (R\pi_{v_1*}\mathscr{L}_{v_1}^\vee)=-1$, and 
$e^{-1}_{T}(R\pi_{v_1*}\mathscr{L}_{v_1}^\vee\otimes{\bf L}_{-1})
=-\ft-{\rm ch}_1(R\pi_{v_1*}\mathscr{L}_{v_1}^\vee)\,.$
Thus
\begin{eqnarray*}
{\rm Cont}(\Gamma_4)
&=&-{1\over 72}\cdot  [B_{v_1}]_{\ft^{-2}}\\
&=&-{1\over 72}\cdot 
\left[{[\overline{\mathcal{M}}_{g,\gamma}(G)^p]^{\rm vir}_{\rm loc}\over e_{T}(R\pi_{v_1*}\mathcal{L}_{v_1}^\vee\otimes{\bf L}_{-1})}\cdot
\big(\prod_{i=1}^2{1\over -\ft-\psi_i}\big)\cdot{1\over -{\ft\over 2}-\psi_3}\right]_{\ft^{-2}}\\
&=&-{2\over 72}
\int_{[\overline{\mathcal{M}}_{0,(1,1,2)}^{1/3, p}]^{\rm vir}_{\rm loc}}1\\
&=&-{1\over 108}\,.
\end{eqnarray*}
Here the last equality follows from genus-zero calculation. 
\end{proof}


\begin{lemma}
\label{lemma-graph5}
For the unstable graph $\Gamma_5$, 
$${\rm Cont}(\Gamma_5)
={4\over 243}\,.$$
\end{lemma}
\begin{proof}
This follows from direct calculation. See Appendix \ref{appendix-msp} for the details.
\end{proof}

It remains to compute the contribution from the last two graphs. We have
\begin{lemma}
\label{lemma-graph67}
For the graphs $\Gamma_6$ and $\Gamma_7$, the contributions are
$${\rm Cont}(\Gamma_6)=-{5\over 486}, \quad {\rm Cont}(\Gamma_7)={1\over 216}.$$
\end{lemma}
\begin{proof}
From the computations in the Appendix \ref{appendix-msp}, we have
$${\rm Cont}(\Gamma_6)={1\over 9}\cdot\left[B_{v_2}\right]_{\ft},
\quad
{\rm Cont}(\Gamma_7)
=-{1\over 9}\cdot \left[B_{v_2}\right]_{\ft}.$$
Recall $\mathbb{E}=\pi_*\omega_{v}$ is the Hodge bundle and $\mathbb{E}^\vee=R^1\pi_*\mathcal{O}_{\sC_v}$ is the dual bundle.
We see
$$e_T(\mathbb{E}_v\otimes \mathbf{L}_{3})=c_1(\mathbb{E})+3\ft, \quad e_T(\mathbb{E}_v^\vee\otimes \mathbf{L}_{-1})=c_1(\mathbb{E}^\vee)-\ft.$$
Assume that $\omega_{(e,v)}=c\ft$ for some constant $c$, then
\begin{eqnarray*}
\left[B_v\right]_{\ft}&=&
\left[\left({e_T(\mathbb{E}_v^\vee\otimes \mathbf{L}_{-1})\over -\ft}\right)^3\cdot{3\ft\over e_T(\mathbb{E}_v\otimes \mathbf{L}_3)}\cdot(-\ft^3)\cdot
\prod_{e\in E_v}{1\over w_{(e,v)}-\psi_{(e,v)}}
\right]_{\ft}\\
&=&\left[{3\ft(c_1(\mathbb{E}^\vee)-\ft)^3\over (c_1(\mathbb{E})+3\ft)(c\ft-\psi_{(e,v)})}\right]_\ft\\
&=&\int_{\overline{\mathcal{M}}_{1,1}}-{1\over c}\left(-{c_1(\mathbb{E})\over 3}-3c_1(\mathbb{E}^\vee)+{\psi_1\over c}\right)\\
&=&{-8c-3\over 72 c^2}\,.
\end{eqnarray*}
The last equality follows from
$$\int_{\overline{\mathcal{M}}_{1,1}}c_1(\mathbb{E})=-\int_{\overline{\mathcal{M}}_{1,1}}c_1(\mathbb{E}^\vee)=\int_{\overline{\mathcal{M}}_{1,1}}\psi_1={1\over 24}\,.$$
Now the result follows from the fact that $c={3\over 2}$ in $\Gamma_6$ and $c=3$ in $\Gamma_7$. 
\end{proof}

Now we complete a proof of Proposition \ref{prop-msp}.
Using Lemma \ref{lemma-graph123}, \ref{lemma-graph4}, \ref{lemma-graph5}, \ref{lemma-graph67}, we have
$$\sum_{i=1}^{7}{\rm Cont}(\Gamma_i)=
0-{1\over 108}+{4\over 243}-{5\over 486}+{1\over 216}
={1\over 648}\,.$$
From \eqref{msp-fjrw} and \eqref{msp-graph-sum}, we obtain the result \eqref{invariant} in Proposition \ref{prop-msp},
$$\Theta_{1,3}=-6 \cdot {\rm Cont}(\Gamma_0)=6\sum_{i=1}^{7}{\rm Cont}(\Gamma_i)={1\over 108}\,.$$

\appendix 
\section{}
\label{sec:appendices}

\subsection{Localization formulae in MSP}
\label{appendix-msp}
We list the formulae needed in Lemma \ref{lemma-graph123}, \ref{lemma-graph4}, \ref{lemma-graph5}, \ref{lemma-graph67}. 
All the edges we consider are in $E_{\infty}$, with $d_e=-{1\over 3}, -{2\over 3}, -1$, or $ -{4\over 3}$.
By Lemma \ref{lem:all-edges} and \eqref{value-Ge}, 
\begin{table}[H] 
  \centering
  \renewcommand{\arraystretch}{1.5} 
 \begin{tabular}{|c|c|c|c|c|}
 \hline
$d_e$ & $-{1\over 3}$ & $-{2\over 3}$ & $-1$ & $-{4\over 3}$\\
\hline
$A_e$ & ${1\over 3\ft}$ & $ {2\over 9\ft^2}$& $-{4\over 27\ft^{3}}$& ${4^{{2}}\over 4!3^5\ft^2}$\\
\hline
$G_e$ & $1$& $2$& $2$& $4$\\
\hline
\end{tabular}
\end{table}
For a vertex $v\in V_{\infty}^S\cup V_{\infty}^{0,2}\cup V_{\infty}^{0,1}\cup V_{1}^{0,1} \cup V_{1}^{0,2}\cup V_{1}^{S}$, we have
\begin{align*}
B_v=
\begin{cases}
\displaystyle{ w_{(e,v)}}, & v\in V_{\infty}^{0,1};\\
\displaystyle{ (3\ft)\cdot w_{(e,v)}}, & v\in V_{1}^{0,1};\\
\displaystyle{ {(-\ft)^{4\epsilon(d_e)}\over w_{(e,v)}+w_{(e',v)}}}, & v\in V_{\infty}^{0,2};\\
\displaystyle{ {-3\ft^4\over w_{(e,v)}+w_{(e',v)}}}, & v\in V_{1}^{0,2};\\
\displaystyle{ {1\over  e_{T}(R\pi_{v_1*}\mathcal{L}_{v_1}^\vee\otimes{\bf L}_{-1})}
\prod\limits_{e\in E_v}{1\over w_{(e,v)}-\psi_{(e,v)}}}, & v\in V_{\infty}^{S};\\
\displaystyle{ ({e_T(\mathbb{E}_v^\vee\otimes \mathbf{L}_{-1})\over -\ft})^3\cdot{3\ft\over e_T(\mathbb{E}_v\otimes \mathbf{L}_3)}\cdot(-\ft^3)
\prod\limits_{e\in E_v}{1\over w_{(e,v)}-\psi_{(e,v)}}}, & v\in V_{1}^{S}.
\end{cases}
\end{align*}

Here we recall that in Lemma \ref{lem:V-zero-two},  $\ep(d_e)=1$ when $d_e\in \ZZ$, and $\ep(d_e)=0$ otherwise. 
And
\begin{align*}
\begin{cases}
\displaystyle{ w_{(e,v)}={3\ft\over -2}, w_{(e,v')}={3\ft\over 2}}, & v\in  V_{\infty}^{0,1}, d_e=-1\,;\\
\displaystyle{ w_{(e,v)}={\ft\over 3d_e}, w_{(e,v')}=-{\ft\over d_e}}, & v \in V_{\infty}\backslash V_{\infty}^{0,1}\,.\\
\end{cases}
\end{align*}

Using these formulae, we list the contributions of the graphs $\Gamma_{k}, \, k=1,\cdots ,7$ below. 
Note that $|{\rm Aut}(\Gamma_4)|=2$, and $|{\rm Aut}(\Gamma_k)|=1$ if $k\neq 4$.
We have
\begin{eqnarray*}
{\rm Cont}(\Gamma_1)
&=&
\left[\left(\prod_vB_v\right)\cdot{1\over \prod_e |G_e|}\cdot\left(\prod_e{A_e}\right)\right]_{\ft^{-1}}
=\left[ \left(B_{v_1}\cdot (3\ft)\cdot (-{\ft\over -{4\over 3}})\right)\cdot{1 \over 4}\cdot \left({4^2\over 4!3^5\ft^2}\right)\right]_{\ft^{-1}}\\
&=&{1\over 648}\cdot [B_{v_1}]_{\ft^{-1}}\,,\\
&&\\
{\rm Cont}(\Gamma_2)
&=&
\left[\left(B_{v_1}\cdot{3\ft\over -2}\cdot{-3\ft^4\over-{\ft\over {-{1\over 3}} }+{3\ft\over 2}}\right)
\cdot {1\over 1\cdot2}
\cdot \left({1\over 3\ft}\cdot(-{4\over 27\ft^{3}})\right)
\right]_{\ft^{-1}}\\
&=&-{2\over 81}\cdot [B_{v_1}]_{\ft^{-1}}\,,\\
&&\\
{\rm Cont}(\Gamma_3)
&=&
\left[
\left(
B_{v_1}
\cdot{(-\ft)^{4\epsilon({-{1\over 3}})}\over {\ft\over 3\cdot {(-{1\over 3}})}+{\ft\over 3\cdot {(-{1\over 3}})}}
\cdot{-3\ft^4\over-{\ft\over -{1\over 3}}-{\ft\over -{1\over 3} }}
\cdot(3\ft)(-{\ft\over {-{2\over 3}}})
\right)
\cdot{1\over 1\cdot 1\cdot2}
\cdot \left(({1\over 3\ft})^2\cdot{2\over 9\ft^2}\right)
\right]_{\ft^{-1}}\\
&=&{3\over 162}\cdot [B_{v_1}]_{\ft^{-1}}\,,\\
&&\\
{\rm Cont}(\Gamma_4)
&=&
{1\over |{\rm Aut}(\Gamma_4)|}
\left[
\left(B_{v_1}
\cdot{-3\ft^4\over-{\ft\over -{1\over 3} }-{\ft\over -{1\over 3}}}
\cdot(3\ft)(-{\ft\over {-{2\over 3}}})
\right)
\cdot {1\over 1\cdot 1\cdot 2}
\cdot \left(({1\over 3\ft})^2\cdot{2\over 9\ft^2}\right)
\right]_{\ft^{-1}}\\
&=&-{1\over 72}\cdot \left[B_{v_1}\right]_{\ft^{-2}}\,,\\
&&\\
{\rm Cont}(\Gamma_5)&=&
\left[
\left({(-\ft)^{4\epsilon({-{1\over 3}})}\over {\ft\over 3\cdot({-{1\over 3}})}+{\ft\over 3\cdot({-{2\over 3}})}}
\cdot{-3\ft^4\over-{\ft\over -{1\over 3} }-{\ft\over -{2\over 3} }}\right)
\cdot {1\over 1\cdot 2}
\cdot \left({1\over 3\ft}\cdot{2\over 9\ft^2}\right)
\right]_{\ft^{-1}}\\
&=&{4\over 243}\,,\\
&&\\
{\rm Cont}(\Gamma_6)&=&
\left[
\left({3\ft\over -2}\cdot B_{v_2}\right)
\cdot{1\over 2}
\cdot\left(-{4\over 27\ft^{3}}\right)
\right]_{\ft^{-1}}\\
&=&{1\over 9
}\cdot \left[B_{v_2}\right]_{\ft}\,,\\
&&\\
{\rm Cont}(\Gamma_7)&=&
\left[
\left(
{(-\ft)^{4\epsilon({-{1\over 3}})}\over {\ft\over 3\cdot({-{1\over 3}})}+{\ft\over 3\cdot({-{2\over 3}})}}
\cdot B_{v_2}
\cdot(3\ft)(-{\ft\over {-{2\over 3}}})\right)
\cdot{1\over1\cdot 2}
\cdot\left({1\over 3\ft}\cdot{2\over 9\ft^2}\right)
\right]_{\ft^{-1}}\\
&=&-{1\over 9
}\cdot\left[B_{v_2}\right]_{\ft}\,.
\end{eqnarray*}



\subsection{Comparison with the maximal group case}
We consider the LG pair $$(W=x_1^3+x_2^3+x_3^3, G=G_{\rm max}).$$ 
Here $G_{\rm max}$ is the group of all diagonal symmetries of $W$, that is $G_{\rm max}\cong(\mu_3)^3$. The FJRW invariants at any genus have been studied for this LG pair in \cite{KS11}, as its underlying Frobenius manifold is generically semisimple. 
We compute a genus-one FJRW invariant for this pair here, its derivation is much simpler than that for the genus-one invariant we computed for the pair $(x_1^3+x_2^3+x_3^3, \mu_3).$

We consider the morphism
$$f:\Mbar_{1,(2^3)}(W, (\mu_3)^3)\longrightarrow \Mbar_{1,(2^3)}(x_1^3,\mu_3)\times\Mbar_{1,(2^3)}(x_2^3,\mu_3)\times\Mbar_{1,(2^3)}(x_3^3,\mu_3)$$
which sends $(\cC,\Sigma,\cL_1,\cL_2, \cL_3)$ to $\bl(\cC,\Sigma,\cL_1),(\cC,\Sigma,\sL_2), (\cC,\Sigma,\cL_3)\br$.
Applying \cite[Thm 4.11]{CLL}, 
$$
[\Mbar_{1,(2^3)}(W,(\mu_3)^3)^p]^\virt= f^* ( [\Mbar_{1,(2^3)}(x_1^3,\mu_3)^p]^\virt\times  [\Mbar_{1,(2^3)}(x_2^3,\mu_3)^p]^\virt\times  [\Mbar_{1,(2^3)}(x_3^3,\mu_3)^p]^\virt).
$$
According to \cite{FJR13} and \cite{CLL}, the cohomology class $\Lambda_{1,3}^{(W, G_{\rm max})}(\one_{J^2}, \one_{J^2}, \one_{J^2})$ is
$$
{|G|^{g}\over \deg {\rm st}}{\rm PD}\ {\rm st}_*\left((-1)^{D}[\Mbar_{1,(2^3)}(W,(\mu_3)^3)^p]^\virt\cap [\Mbar_{1,2^3}(W,(\mu_3)^3)]\right).$$
Since $\deg{\rm st}=|G|^{2g-1}$ and $g=1$ in this case, 
applying the projection formula to \eqref{mixed-terms} and then \eqref{g=1-kappa}, we obtain
\begin{eqnarray*}
{\rm st}_*\left[\overline{\mathcal{M}}_{1, (2^3)}^{1/3,p}\right]^{\rm vir}
&=&{1\over 12}\cdot (\kappa_1-\sum_{i=1}^{3}\psi_i)-{\delta_{\rm irr}\over 36}-{\delta_{0,2}\over 4}+{\delta_{0,3}\over 12}-{\delta_{0,3}\over 3}\nonumber\\
&=&-{\delta_{\rm irr}\over 36}-{\delta_{0,2}\over 3}-{\delta_{0,3}\over 3}\quad
\in A^1(\Mbar_{1,3})\,.
\end{eqnarray*}
By Table \ref{table-cubic}, 
we have
\begin{eqnarray*}
\langle\one_{J^2}, \one_{J^2}, \one_{J^2}\rangle_{1,3}^{(W, G_{\rm max})}&=&
\left({\rm st}_*\left[\overline{\mathcal{M}}_{1, (2^3)}^{1/3,p}\right]^{\rm vir}\right)^3=
\left(-{\delta_{\rm irr}\over 36}-{\delta_{0,2}\over 3}-{\delta_{0,3}\over 3}\right)^3
={1\over 324}.
\end{eqnarray*}

\bibliographystyle{amsalpha}

 \bigskip{}

\noindent{\small Shanghai Center for Mathematical Sciences, Fudan University, Shanghai 200438, China;}

\noindent{\small Department of Mathematics, Stanford University, Stanford, CA 94305, USA
}

\noindent{\small E-mail: \tt jli@math.stanford.edu}

\medskip{}

\noindent{\small 
Department of Mathematics,
Hong Kong University of Science and Technology,
Clear Water Bay, Kowloon,
Hong Kong}

\noindent{\small E-mail: \tt  mawpli@ust.hk}

\medskip{}

\noindent{\small Department of Mathematics, University of Oregon, Eugene, OR 97403, USA}

\noindent{\small E-mail: \tt yfshen@uoregon.edu}

\medskip{}

\noindent{\small Yau Mathematical Sciences Center,  Tsinghua University, Beijing 100084, China}

\noindent{\small E-mail: \tt jzhou2018@mail.tsinghua.edu.cn}

\end{document}